%% file: main.tex
\documentclass{amsart}[10pt]
\usepackage[utf8]{inputenc}

\input{Preamble}

\graphicspath{{}}
\usetikzlibrary{patterns}

\title{On the Zeros of the Miller Basis of Cusp Forms}
\author{Roei Raveh}
\address{School of Mathematical Sciences\\
Tel Aviv University\\
Tel Aviv 69978, Israel}
\email{roeiraveh@mail.tau.ac.il}
\date{}

\begin{document}
\begin{abstract}
    We study the zeros of cusp forms in the Miller basis whose vanishing order at infinity is a fixed number $m$. We show that for sufficiently large weights, the finite zeros of such forms in the fundamental domain, all lie on the circular part of the boundary of the fundamental domain. We further show and quantify an effective bound for the weight, which is linear in terms of $m$.
\end{abstract}
\maketitle
\section{Introduction}
Let $k\ge0$ be an even integer, and let $M_{k}$ denote the linear space of holomorphic modular forms of weight $k$ for the modular group $\Gamma = \psl $. Each modular form $f$ has a $q$-expansion (i.e.\  Fourier series),
\[f(\tau) = \sum_{n=n_{0}}^{\infty}a_{f}\pr{n}q^{n}, \quad q= e^{2\pi i\tau}\]
where $\tau\in\HH$, and $n_{0}=\ord_{\infty}\pr{f} $.\\

For any nonzero $f\in M_{k}$, we have the valence formula (see \cite{serre}):
\begin{equation}\label{valence}
    \ord_{\infty}\pr{f} + \frac{1}{2}\ord_{i}\pr{f} + \frac{1}{3}\ord_{\rho}\pr{f} + \sum_{z\in\cF\smallsetminus\{i,\rho\}} \ord_{z}\pr{f} =\frac{k}{12},
\end{equation}
where $\rho = e^{\sfrac{2\pi i}{3}} = -\frac{1}{2}+\frac{\sqrt{3}}{2}i$, and $\cF$ is the fundamental domain
\[\cF = \prs{z\in\HH:\abs{z}\ge 1, -\frac{1}{2}\le \Re\pr{z}\le 0}\bigcup \prs{z\in\HH:\abs{z}>1, 0<\Re\pr{z}<\frac{1}{2}},\]
as seen in Figure \ref{fig: fundomain}.
\begin{figure}[ht]
    \centering
    \begin{tikzpicture}[scale = 2] 
    \filldraw (0,0) circle (0.75pt) node[below] {$0$};
    \filldraw (1,0) circle (0.75pt) node[below right] {$1$};
    \filldraw (-1,0) circle (0.75pt) node[below left] {$-1$};
    \filldraw (-0.5,0) circle (0.75pt) node[below] {$-\frac{1}{2}$};
    \filldraw (0.5,0) circle (0.75pt) node[below] {$\frac{1}{2}$};
    \draw[thick,->] (-1.25,0) -- (1.25,0) node[right] {};
    \draw[fill=gray!30,draw =white] (-0.5,2.6) -- (-0.5,0.866) arc (120:60:1) -- (0.5,2.6);
    \draw[line width=0.15mm,loosely dashed] (1,0) arc (0:60:1);
    \draw[line width=0.15mm,dashed] (0.5,0.866) arc (60:90:1);
    \draw[line width=0.2mm] (0,1) arc (90:120:1);
    \draw[line width=0.15mm,loosely dashed](-0.5,0.866) arc (120:180:1);
    \draw[line width=0.2mm] (-0.5,0.866) -- (-0.5,2.6);
    \draw[loosely dashed] (0.5,0) -- (0.5,.866);
    \draw[dashed] (.5,.866) -- (0.5, 2.6);
    \draw[loosely dashed] (-0.5,0) -- (-0.5,.866);
    \filldraw (0,1) circle (0.75pt) node[above] {$i$};
    \filldraw (-0.5,0.866) circle (0.75pt) node[left] {$\rho$};
    \filldraw[gray!30] (0,1.8) circle (0.75pt) node {\color{black}$\cF$};
    \end{tikzpicture}
    \caption{The fundamental domain $\cF$}
    \label{fig: fundomain}
\end{figure}

Writing $k=12\ell +k'$ where $k'\in\{0,4,6,8,10,14\}$, we obtain
\begin{equation}
    \ell \ge \ord_{\infty}\pr{f}.
\end{equation}
Those formulas provide a powerful tool for studying modular forms via their zeros. The valence formula \eqref{valence} implies that for a nonzero modular form of weight $k$, there are $\frac{k}{12}-\ord_{\infty}(f)+O\pr{1}$ zeros in the fundamental domain $\cF$. The space $M_{k}$ is finite-dimensional and is spanned by the Eisenstein series \eqref{esn} and the space of cusp forms $S_{k}$, with $\dim S_{k}=\ell$. The zeros of the Eisenstein series were studied in 1970 by F. Rankin and P. Swinnerton-Dyer \cite{sdr}. In their paper, they proved that all the zeros of the Eisenstein series in the fundamental domain lie on the arc $\cA=\left\{e^{i\theta}:\frac{\pi}{2}\le \theta\le \frac{2\pi}{3}\right\}$ and become uniformly distributed in $\cA$ as $k\to\infty$. This argument of Rankin and Swinnerton-Dyer was used to prove similar results, for instance, in the work of R. Rankin on the zeros of Poincar{\'e} series \cite{rankin} and in the work of S. Gun on the zeros of certain linear combinations of Poincar{\'e} series \cite{gun}. For different types of results about zeros of various modular forms, see \cite{holowinsky2010mass}, \cite{kimmel}, \cite{RVY}, \cite{Ringeling}, \cite{rudnick2005}, \cite{rudnick2023}, \cite{Xue-Zhu}.\\

This paper will discuss the zeros of the Miller basis of modular forms. The elements of the Miller basis $\left\{g_{k,m}\right\}_{m=0}^{\ell}$ are uniquely defined for every $0\le m\le\ell$ by requiring \[g_{k,m}(\tau)=q^{m}+O\left(q^{\ell+1}\right).\]
The Miller basis forms a canonical basis of $M_{k}$ in the sense that it is a basis of reduced row echelon form.
In particular, we are interested in the cusp forms of the Miller basis $g_{k,1},\ldots,g_{k,\ell}$.\\

 W. Duke and P. Jenkins \cite{dukejen} showed that the \say{gap form} $g_{k,0} = 1+O\pr{q^{\ell+1}}$ has all its zeros on the arc $\cA$ in the fundamental domain. This is not generally true for $g_{k,m}$ with $m\ge 1$; for example, as pointed out in \cite{dukejen}, $g_{132,9}$ doesn't have all its zeros on the arc $\cA$. Nevertheless, an asymptotic result can be achieved:
\begin{thm}\label{mr}
    Fix $m\ge 1$. There exist $\alpha,\beta>0$ so that if $\ell > \alpha m+\beta$, then all the zeros of $g_{k,m}$ in the fundamental domain lie on the arc $\left\{e^{i\theta}:\frac{\pi}{2}\le \theta\le \frac{2\pi}{3}\right\}$, and become uniformly distributed on the arc as $k\to \infty$.
\end{thm}

In \S \ref{background} we discuss some needed background on modular forms and the Miller basis.
In \S \ref{smr} we will prove Theorem \ref{mr}.
In \S \ref{What are the bounds} we quantify those bounds and show we can choose $\alpha = 4.5$ and $\beta=9.5$.
Finally, in \S \ref{m=1} we investigate the behavior of $g_{k,1}$, and prove the following:
\begin{thm}\label{the case m=1}
    For every $\ell\ge 1$, all the zeros of $g_{k,1}(\tau)=q+O\pr{q^{\ell+1}}$ in the fundamental domain lie on the arc. 
\end{thm}

\section{Background and Preliminaries on Modular Forms}\label{background}
\subsection{Definitions}
Let $k\ge 0$ be an even integer, and let $\HH=\{\tau:\Im(\tau)>0\}$ denote the upper half plane. Let $f:\HH\to\C$ be a holomorphic function; we say that $f$ is a \emph{modular form of weight $k$} if
\begin{equation}\label{def1}
    f\left(\frac{a\tau+b}{c\tau +d}\right)=\left(c\tau+d\right)^{k}f(\tau),\quad\forall \begin{psmallmatrix}
    a & b \\
    c & d
    \end{psmallmatrix}\in\psl.
\end{equation}
and $f$ is bounded as $\Im(\tau)\to\infty$. If $f$ vanishes as $\Im(\tau)\to\infty$, we say that $f$ is a cusp form.
\begin{remark}
We can replace \eqref{def1} with the following conditions:
\begin{align}
    f(\tau) & = f(\tau+1),\label{eq:1periodic}\\
    f(\tau) & = \tau^{-k}f(-1/\tau). \label{eq:inverting}
\end{align}
\end{remark}
When $k\ge 4$ and even, there exists a nonzero modular form in $M_{k}$ known as the (normalized) Eisenstein series
\begin{equation}\label{esn}
    E_{k}(\tau) = \frac{1}{2}\msum{(m,n)\in\Z^{2}}{\gcd(m,n)=1}\frac{1}{\left(m\tau+n\right)^{k}}= 1 -\gamma_{k}\sum_{n=1}^{\infty}\sigma_{k-1}(n)q^{n},
\end{equation}
where $\sigma_{k-1}(n)=\sum_{d\mid n}d^{k-1}$, $\gamma_{k}=\frac{2k}{B_k}$, and $B_k$ is the $k$-th Bernoulli number.\\
One can also define the Eisenstein series of weight $2$
\begin{equation}
    E_{2}(\tau) =\frac{1}{2\zeta\pr{2}}\sum_{n\neq 0}\frac{1}{n^{2}} + \frac{1}{2\zeta\pr{2}}\sum_{m\neq 0}\sum_{n\in \Z}\frac{1}{\pr{m\tau+n}^2}= 1 - 24\sum_{n=1}^{\infty}\sigma_{1}(n)q^{n}.
\end{equation}
While $E_{2}$ is not a modular form, it has some modular properties and transforms as
\begin{equation}\label{eq: E_2 transform}
    E_{2}\pr{-1/\tau} = \tau^{2}E_{2}\pr{\tau}+\frac{6\tau}{i\pi}.
\end{equation}
We will also define $E_{0}\pr{\tau}=1$.
The first cusp form we encounter is the Modular Discriminant,
\begin{equation}\label{moddisc}
    \Delta(\tau) = \frac{1}{1728}\left(E_{4}^{3}(\tau)-E_{6}^{2}(\tau)\right) = q\prod_{n=1}^{\infty}\left(1-q^{n}\right)^{24} = \sum_{n=1}^{\infty}\tau(n)q^{n}.
\end{equation}
The coefficients $\tau(n)$ are known as Ramanujan's tau function and are all integers.\\
Lastly, there is a meromorphic modular form of weight $0$, Klein's absolute invariant, or the $j$-function:
\begin{equation}\label{jinv}
    j(\tau) = \frac{E_{4}^{3}(\tau)}{\Delta(\tau)} = q^{-1} + 744 + \sum_{n=1}^{\infty}c(n)q^n,
\end{equation}
where the coefficients $c(n)$ are all positive integers.\\
\subsection{The Miller Basis for Modular Forms} Let $m\in\{1,\ldots,\ell\}$, and denote:
\[e_{k,m}=\Delta^{\ell}E_{k'}j^{\ell-m},\]
Notice that $e_{k,m}$ has integer coefficients and that
\[e_{k,m}= \left(q^\ell + O\left(q^{\ell+1}\right)\right)\left(q^{-\ell+m}+O\left(q^{-\ell+m+1}\right)\right)=q^{m} + O\left(q^{m+1}\right).\]
Using Gaussian elimination we obtain a basis of reduced row echelon form, $\left\{g_{k,m}\right\}_{m=1}^{\ell}$. Moreover, for any $m\in\{1,\ldots,\ell\}$ there exists a unique polynomial $F_{k,m}\in \Z[x]$ of degree $\ell-m$, so that
\begin{equation}
    g_{k,m}=\Delta^{\ell}E_{k'}F_{k,m}(j)=q^{m}+O\left(q^{\ell+1}\right).
\end{equation}
\begin{remark}
Following each step in the Gaussian elimination process, we can see that $F_{k,m}$ has integer coefficients. The polynomial $F_{k,m}$ is the associated Faber polynomial of $g_{k,m}$; Faber polynomials play a major role in the study of zeros of modular forms (see \cite{rudnick2023}), and we will discuss those in detail in \S\ref{FaberPoly}.
\end{remark}
\subsection{Modular Forms on the Arc and Under Conjugation}\label{forms on the arc}
Let $f\in M_{k}$, and consider its $q$-expansion $f(\tau) = \sum_{n=0}^{\infty}a_{n}q^{n}$. Suppose $a_{n}$ are all real, then:
\begin{equation}
    \overline{f(\tau)}=\sum_{n=0}^{\infty}a_{n}\overline{e^{2\pi in\tau}}
    =\sum_{n=0}^{\infty}a_{n}e^{2\pi in(-\overline{\tau})} = f(-\overline{\tau}).
\end{equation}
Suppose $\tau = e^{i\theta}$, with $\theta\in\left[\frac{\pi}{2},\frac{2\pi}{3}\right]$. We have $\overline{\tau} = 1/e^{i\theta}$ and from \eqref{eq:inverting} we get
\begin{equation*}
    \overline{f\left(e^{i\theta}\right)} = f\left(-1/e^{i\theta}\right) = e^{ik\theta}f\left(e^{i\theta}\right),
\end{equation*}
which yields, 
\begin{equation}
    \overline{e^{ik\theta/2}f\left(e^{i\theta}\right)}=e^{-ik\theta/2}\overline{f\left(e^{i\theta}\right)} = e^{ik\theta/2}f\left(e^{i\theta}\right).
\end{equation}
So, $g(\theta)=e^{ik\theta/2}f(e^{i\theta})$ is real valued.

\section{Proof of Theorem~\ref{mr}}\label{smr}
We begin with proving some bounds on $\Delta$.
\subsection{Bounds and Properties of $\Delta$}
The goal of this section is to prove the following proposition:
\begin{prop}\label{Prop: Bounds of delta on the arc}
    For all $\theta\in\left[\frac{\pi}{2},\frac{2\pi}{3}\right]$,
    \[\abs{\Delta\pr{e^{i\theta}}}\ge \abs{\Delta(i)} = \pr{\frac{\varpi}{\sqrt{2}\pi}}^{12} = 0.00178537\ldots\]
    with $\varpi = 2\int_{0}^{1}\frac{dx}{\sqrt{1-x^{4}}} = 2.622057\ldots $, and 
    \[\abs{\Delta\pr{e^{i\theta}}}\le \abs{\Delta(\rho)} = \frac{27}{256}\pr{\frac{\varpi'}{\pi}}^{12} =  0.00480514\ldots\]
    with $\varpi' = 2\int_{0}^{1}\frac{dx}{\sqrt{1-x^{6}}} =2.42865\ldots$.
\end{prop}
As this proposition implies, on the interval $\prb{\frac{\pi}{2},\frac{2\pi}{3}}$ the function $\theta\mapsto\abs{\Delta\pr{e^{i\theta}}}$ attains its extrema values on the boundary of the interval.
To prove this proposition, we will need the following lemmata:
\begin{lemma}\label{lemma: e_2,e_4, delta are negative}
    \begin{enumerate}[label = (\roman*)]
        \item The function $e_{4}:\prb{\frac{\pi}{2},\frac{2\pi}{3}}\to\R$ defined by $e_{4}\pr{\theta}=e^{2i\theta}E_{4}\pr{e^{i\theta}}$ for all $\theta\in\prb{\frac{\pi}{2},\frac{2\pi}{3}}$ is negative for all $\theta\in\left[\frac{\pi}{2},\frac{2\pi}{3}\right)$ and vanishes at $\frac{2\pi}{3}$. \label{e_{4} is negative}
        \item The function $\delta:\prb{\frac{\pi}{2},\frac{2\pi}{3}}\to\R$ defined by $\delta(\theta)=e^{6i\theta}\Delta\pr{e^{i\theta}}$ for all $\theta\in\prb{\frac{\pi}{2},\frac{2\pi}{3}}$ is negative for all $\theta\in\left[\frac{\pi}{2},\frac{2\pi}{3}\right]$.
        \item The function $e_{2}:\prb{\frac{\pi}{2},\frac{2\pi}{3}}\to\R$ defined by $e_{2}\pr{\theta} = e^{i\theta}E_{2}\pr{e^{i\theta}}+\frac{3}{i\pi}$  for all $\theta\in\prb{\frac{\pi}{2},\frac{2\pi}{3}}$ is real valued and is negative for all $\theta\in\left(\frac{\pi}{2},\frac{2\pi}{3}\right]$.
    \end{enumerate}
\end{lemma}
\begin{proof}
    \begin{enumerate}[label = (\roman*)]
        \item $E_{4}$ has a unique zero in the fundamental domain, at $\rho = e^{2\pi i/3}$. Therefore, $e_{4}$ is real-valued, continuous, and nonzero for all $\theta\in\left[\frac{\pi}{2},\frac{2\pi}{3}\right)$. Hence, it is enough to show that $e_{4}\pr{\frac{\pi}{2}}<0$, and indeed
        \[-e_{4}\pr{\frac{\pi}{2}}=E_{4}\pr{i}=1+240\sum_{n=1}^{\infty}\sigma_{3}\pr{n}e^{-2\pi n}\ge 1>0.\]
        \item We know that $\Delta$ never vanishes, thus $\delta$ never vanishes. Hence, it is enough to show that $\delta\pr{\frac{\pi}{2}}<0$. Indeed,
        \[\delta\pr{\frac{\pi}{2}} = -\Delta(i)=-\frac{E_{4}\pr{i}^{3}}{1728}<0,\]
        since $E_{4}(i)>0$. 
        \item First, we will show that $e_{2}$ is real valued. Since $E_{2}$ has real Fourier coefficients, we have $\overline{E_{2}(\tau)}=E_{2}\pr{-\overline{\tau}}$. Using \eqref{eq: E_2 transform} we obtain
        \begin{multline*}
            \overline{e_{2}(\theta)} = e^{-i\theta}\overline{E_{2}\pr{e^{i\theta}}} - \frac{3}{i\pi}= e^{-i\theta}E_{2}\pr{-e^{-i\theta}}- \frac{3}{i\pi} \\ = e^{-i\theta}\pr{e^{2i\theta}E_{2}\pr{e^{i\theta}}+\frac{6e^{i\theta}}{i\pi }} - \frac{3}{i\pi} = e^{i\theta}E_{2}\pr{e^{i\theta}}+ \frac{3}{i\pi} = e_{2}\pr{\theta}.
        \end{multline*}
        Using \eqref{eq: E_2 transform} again, we get $E_{2}\pr{i}=\frac{3}{\pi}$, and thus
        \[e_{2}\pr{\frac{\pi}{2}}=e^{i\pi/2}E_{2}\pr{i} + \frac{3}{i\pi}=i\frac{3}{\pi}-\frac{3}{\pi}i=0.\]
        We claim that $e_{2}$ is decreasing on $\left[\frac{\pi}{2},\frac{2\pi}{3}\right]$, which yields $e_{2}(\theta)<0$ for all $\theta\in \left(\frac{\pi}{2},\frac{2\pi}{3}\right]$. We will show that $\frac{de_{2}}{d\theta}<0$:\\
        Using an identity of Ramanujan \cite{ramid}, we know that
        \[\frac{1}{2\pi i}\frac{dE_{2}}{d\tau}=\frac{E_{2}^{2}-E_{4}}{12}.\]
        Hence,
        \begin{align*}
            \frac{de_{2}}{d\theta}(\theta) & = \frac{d}{d\theta}\pr{e^{i\theta}E_{2}\pr{e^{i\theta}}+\frac{3}{i\pi}}\\
            & = ie^{i\theta}E_{2}\pr{e^{i\theta}} + e^{i\theta} \frac{d}{d\theta}\pr{E_{2}\pr{e^{i\theta}}} \\
            & = i\pr{e_{2}\pr{\theta}-\frac{3}{i\pi}} + e^{i\theta}\cdot ie^{i\theta}\frac{dE_{2}}{d\tau}\pr{e^{i\theta}} \\
            & = ie_{2}\pr{\theta}-\frac{3}{\pi}- e^{2i\theta}\frac{\pi}{6}\pr{E_{2}\pr{e^{i\theta}}^{2} - E_{4}\pr{e^{i\theta}}}\\
            & = ie_{2}\pr{\theta}-\frac{3}{\pi}-\frac{\pi}{6}\pr{\pr{e^{i\theta}E_{2}\pr{e^{i\theta}}}^{2} - e_{4}\pr{\theta}}\\
            & = ie_{2}\pr{\theta}-\frac{3}{\pi}+ \frac{\pi}{6}e_{4}\pr{\theta}-\frac{\pi}{6}\pr{e_{2}\pr{\theta}-\frac{3}{i\pi}}^{2}\\
            & = ie_{2}\pr{\theta}-\frac{3}{\pi}+ \frac{\pi}{6}e_{4}\pr{\theta}-\frac{\pi}{6}e_{2}\pr{\theta}^{2}+\frac{e_{2}\pr{\theta}}{i}+\frac{3}{2\pi}\\
            & = -\frac{3}{2\pi}-\frac{\pi}{6}e_{2}\pr{\theta}^{2}+\frac{\pi}{6}e_{4}\pr{\theta}.
        \end{align*}
        Since $e_{4}\pr{\theta}<0$ for all $\theta\in\left[\frac{\pi}{2},\frac{2\pi}{3}\right)$ we obtain $\frac{de_{2}}{d\theta}\pr{\theta}<0$ for all $\theta \in \left(\frac{\pi}{2},\frac{2\pi}{3}\right]$.
    \end{enumerate}
\end{proof}
We now have the tools to prove the following lemma, which is key for proving Proposition \ref{Prop: Bounds of delta on the arc},
\begin{lemma}
    The function $\delta$ is decreasing on $\left[\frac{\pi}{2},\frac{2\pi}{3}\right]$.
\end{lemma}
\begin{proof}
    As stated in \cite{ramid},
    \[\frac{d\Delta}{d\tau} = 2\pi i E_{2}\Delta.\]
    Hence,
    \begin{multline*}
        \frac{d\delta}{d\theta}\pr{\theta} = 6ie^{6i\theta}\Delta\pr{e^{i\theta}}+e^{6i\theta}\cdot ie^{i\theta}\frac{d\Delta}{d\tau}\pr{e^{i\theta}} = 6i\delta\pr{\theta}-2\pi e^{6i\theta}\cdot e^{i\theta}E_{2}\pr{e^{i\theta}}\Delta\pr{e^{i\theta}}\\
        = 6i\delta\pr{\theta}-2\pi \pr{e_{2}\pr{\theta}- \frac{3}{i\pi}}\delta\pr{i\theta} =-2\pi e_{2}\pr{\theta}\delta\pr{\theta}.
    \end{multline*}
    Since $e_{2}<0$ for all $\theta\in \left(\frac{\pi}{2},\frac{2\pi}{3}\right]$ and $\delta\pr{\theta}< 0$ for all $\theta\in \left[\frac{\pi}{2},\frac{2\pi}{3}\right]$ we get $\frac{d\delta}{d\theta}\pr{\theta} <0$ for all $\theta \in \left(\frac{\pi}{2},\frac{2\pi}{3}\right]$. Hence, $\delta$ is decreasing on $\left[\frac{\pi}{2},\frac{2\pi}{3}\right]$.
\end{proof}

\begin{figure}[ht]
    \begin{center}
    \includegraphics[height=65mm]{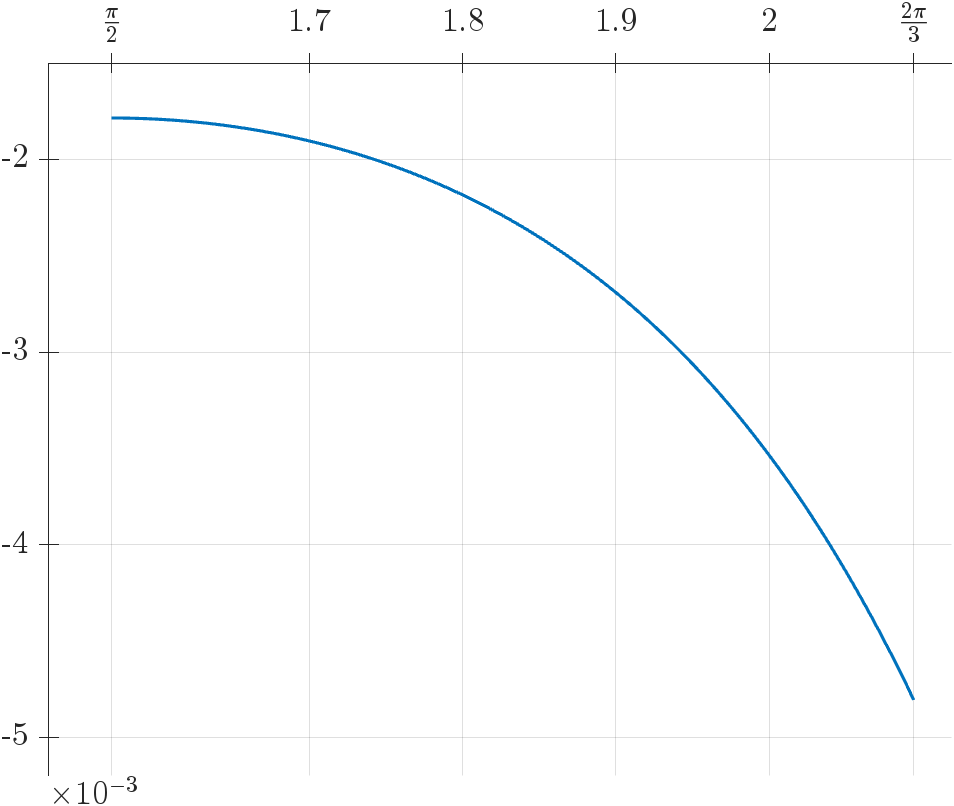}
    \caption{$\delta\pr{\theta}$ on the interval $\left[\frac{\pi}{2},\frac{2\pi}{3}\right]$.}
    \label{fig:Delta on arc}
    \end{center}
\end{figure}

Finally, we can prove Proposition \ref{Prop: Bounds of delta on the arc}:
\begin{proof}[Proof of Proposition \ref{Prop: Bounds of delta on the arc}]
    The function $\abs{\Delta\pr{e^{i\theta}}}$ is increasing on $\left[\frac{\pi}{2},\frac{2\pi}{3}\right]$, thus, for all $\theta\in \left[\frac{\pi}{2},\frac{2\pi}{3}\right]$ we have
    \[\Delta\pr{i} \le \abs{\Delta\pr{e^{i\theta}}}\le \abs{\Delta\pr{\rho}}.\]
    Since $\Delta = \frac{E_{4}^{3}-E_{6}^{2}}{1728}$ and since $E_{4}\pr{\rho}=0$, $E_{6}\pr{i} =0 $ we get
    \begin{align}
        \Delta\pr{i}  & = \frac{E_{4}\pr{i}^3}{1728},\\
        \Delta\pr{\rho}  & = -\frac{E_{6}\pr{\rho}^2}{1728}.
    \end{align}
    Hurwitz  \cite{Hurwitz} showed that 
    \[\sum_{(m,n)\in\Z^{2}\smallsetminus\{(0,0)\}}\frac{1}{\pr{mi+n}^{4}}=\frac{\varpi^{4}}{15},\]
    with $\varpi = 2\int_{0}^{1}\frac{dx}{\sqrt{1-x^{4}}} = 2.622057\ldots $, hence
    \[ E_{4}\pr{i} =\frac{1}{2\zeta(4)} \sum_{(m,n)\in\Z^{2}\smallsetminus\{(0,0)\}}\frac{1}{\pr{mi+n}^{4}}=\frac{3\varpi^{4}}{\pi^{4}}.\]
    Therefore,
    \[\Delta\pr{i}  = \frac{E_{4}\pr{i}^3}{1728}=\pr{\frac{E_{4}\pr{i}}{12}}^{3} = \pr{\frac{\varpi}{\sqrt{2}\pi}}^{12}=0.00178537\ldots.\]
    Katayama \cite{Katayama} gives a generalization of Hurwitz’s formula
    \[\sum_{(m,n)\in\Z^{2}\smallsetminus\{(0,0)\}}\frac{1}{\pr{m\rho+n}^{6}}=\frac{\varpi'^{6}}{35},\]
    with $\varpi' = 2\int_{0}^{1}\frac{dx}{\sqrt{1-x^{6}}} =2.42865\ldots$, hence
    \[E_{6}\pr{\rho} = \frac{1}{2\zeta\pr{6}}\sum_{(m,n)\in\Z^{2}\smallsetminus\{(0,0)\}}\frac{1}{\pr{m\rho+n}^{6}}=\frac{945}{2\pi^{6}}\frac{\varpi'^{6}}{35}=\frac{27\varpi'^{6}}{2\pi^{6}}.\]
    Therefore,
    \[\Delta(\rho)=-\frac{E_{6}\pr{\rho}^{2}}{1728} = -\frac{27}{256}\pr{\frac{\varpi'}{\pi}}^{12} =  -0.00480514\ldots.\qedhere\]
\end{proof}
\begin{prop} \label{tau bounds}
    \begin{enumerate}[label = (\roman*)]
        \item For all $\tau\in\mathbb{H}$ with $e^{-2\pi\Im(\tau) }\le \frac{1}{3}$,
        \[\left|\Delta\left(\tau\right)\right|\ge e^{-2\pi \Im(\tau)}\left(1-e^{-2\pi \Im(\tau)}-e^{-4\pi \Im(\tau)}-\frac{2e^{-4\pi \Im(\tau)}}{1-e^{-2\pi \Im(\tau)}}\right)^{24}.\]
        \item For all $\tau\in\mathbb{H}$
        \[\left|\Delta\left(\tau\right)\right|\le e^{-2\pi \Im(\tau)}\left(1+2e^{-2\pi\Im(\tau)}+\frac{2e^{-8\pi \Im(\tau)}}{1-e^{-4\pi \Im(\tau)}}\right)^{24}.\]
    \end{enumerate}
\end{prop}
\begin{proof}
\begin{enumerate}[label = (\roman*)]
    \item Recall that by Euler's pentagonal number theorem, for any $|z|<1$ we have:
    \[\prod_{n=1}^{\infty}\left(1-z^{n}\right)=1 + \sum_{k=1}^{\infty}(-1)^{k}\left(z^{\frac{k(3k+1)}{2}}+z^{\frac{k(3k-1)}{2}}\right)\]
    and the series on the RHS converges absolutely.
    Due to the absolute convergence, we can change the order of summation and write:
    \begin{align*}
        \prod_{n=1}^{\infty}\left(1-z^{n}\right) & = 1 + \sum_{k=1}^{\infty}\left(z^{k(6k+1)}+z^{k(6k-1)}-z^{(2k-1)(3k-1)}-z^{(2k-1)(3k-2)}\right)\\
        & = 1 +z^7 +z^5 - z - z^2 + \sum_{k=2}^{\infty}z^{6k^{2}}\left(z^{k}+z^{-k}-z^{-5k+1}-z^{-7k+2}\right).
    \end{align*}
    Let $z\in (0,1)$. For all $k\ge 1$, we have $z^{-5k+1}<z^{-7k+2}$ and since $0 \le 5k^2-7k+2$ we also have $z^{6k^{2}-7k+2}\le z^{k^{2}}$.
    Therefore,
    \begin{multline*}
        z^{6k^{2}}\left(z^{k}+z^{-k}-z^{-5k+1}-z^{-7k+2}\right)\ge -z^{6k^{2}}\left(z^{-5k+1}+z^{-7k+2}\right) \\> -2z^{6k^{2}}z^{-7k+2}> -2z^{k^2} \ge -2z^{k}.
    \end{multline*}
    Hence, 
    \[\prod_{n=1}^{\infty}\left(1-z^{n}\right) \ge 1-z-z^2 -2\sum_{k=1}^{\infty}z^{k} = 1-z - z^2 - \frac{2z^2}{1-z}. \]
    Thus, for any $\tau\in\mathbb{H}$ with $e^{-2\pi\Im(\tau) }\le \frac{1}{3}$ we have:
    \begin{align*}
        \left|q\prod_{n=1}^{\infty}\left(1-q^{n}\right)^{24}\right| & \ge|q|\left(\prod_{n=1}^{\infty}\left(1-|q|^{n}\right)\right)^{24}=e^{-2\pi \Im(\tau)}\left(\prod_{n=1}^{\infty}\left(1-e^{-2\pi n\Im(\tau)}\right)\right)^{24}\\
        & \ge e^{-2\pi \Im(\tau)}\left(1-e^{-2\pi \Im(\tau)}-e^{-4\pi \Im(\tau)}-\frac{2e^{-4\pi \Im(\tau)}}{1-e^{-2\pi \Im(\tau)}}\right)^{24}.
    \end{align*}
    
    \item By the pentagonal number theorem and the triangle inequality:
    \begin{multline*}
        \left|\prod_{n=1}^{\infty}\left(1-z^{n}\right)\right|\le 1+\sum_{k=1}^{\infty}\left|(-1)^{k}\left(z^{\frac{k(3k+1)}{2}}+z^{\frac{k(3k-1)}{2}}\right)\right|\\\le 1 + \sum_{k=1}^{\infty}\left|z\right|^{\frac{k(3k+1)}{2}}+\left|z\right|^{\frac{k(3k-1)}{2}}.
    \end{multline*}
    For all $|z|<1$, we have $\left|z\right|^{\frac{k(3k+1)}{2}}<\left|z\right|^{\frac{k(3k-1)}{2}}$. In addition, since $\frac{k(3k-1)}{2}\ge 2k$ for all $k\ge2$, we obtain: $$\left|z\right|^{\frac{k(3k+1)}{2}}+\left|z\right|^{\frac{k(3k-1)}{2}}\le 2\left|z\right|^{\frac{k(3k-1)}{2}}\le 2\left|z\right|^{2k}.$$
    Thus, 
    \[ \left|\prod_{n=1}^{\infty}\left(1-z^{n}\right)\right| \le 1+2\left|z\right|+\sum_{k=2}^{\infty}2\left|z\right|^{2k} = 1 +\left|z\right|+\frac{2\left|z\right|^4}{1-\left|z\right|^2},\]
    which shows \[\left|\Delta\left(\tau\right)\right|\le e^{-2\pi \Im(\tau)}\left(1+2e^{-2\pi\Im(\tau)}+\frac{2e^{-8\pi \Im(\tau)}}{1-e^{-4\pi \Im(\tau)}}\right)^{24},\]
    for all $\tau\in\mathbb{H}$, and concludes our proof.\qedhere
\end{enumerate}
\end{proof}

From Proposition \ref{Prop: Bounds of delta on the arc} and Proposition \ref{tau bounds}, we get the following corollary:
\begin{corr}
For all $\theta \in [\pi/2,2\pi/3]$, $x\in[-1/2,1/2]$:
\begin{equation}\label{deltabounds65}
\left|\frac{\Delta(e^{i\theta})}{\Delta(x+0.65i)}\right|< \frac{1}{2},
\end{equation}
and
\begin{equation}\label{deltabounds75}
\left|\frac{\Delta(e^{i\theta})}{\Delta(x+0.75i)}\right|< \frac{7}{10}.
\end{equation}
\end{corr}
\begin{proof}
Substitute $\tau = x+0.65i$,
\[|\Delta (x+0.65i)|\ge e^{-\frac{13\pi}{10}}\left(1-e^{-\frac{13\pi}{10}}-e^{-\frac{13\pi}{5}}-\frac{2e^{-\frac{26\pi}{5}}}{1-e^{-\frac{13\pi}{10}}}\right)^{24}>0.01\]
Now, using Proposition \ref{Prop: Bounds of delta on the arc}, we know $\left|\Delta(e^{i\theta})\right|<0.005$ for all $\theta\in\left[\frac{\pi}{2},\frac{2\pi}{3}\right]$.
For all  $\theta \in [\pi/2,2\pi/3]$, $x\in[-1/2,1/2]$:
\[\left|\frac{\Delta(e^{i\theta})}{\Delta(x+0.65i)}\right|< \frac{0.005}{0.01}= \frac{1}{2}.\]
Similarly, we can get:
\begin{equation*}
\left|\frac{\Delta(e^{i\theta})}{\Delta(x+0.75i)}\right|< \frac{7}{10}.\qedhere
\end{equation*}
\end{proof}

\subsection{A Key Proposition}
For the proof of Theorem \ref{mr} we will also need the following proposition:

\begin{prop}\label{mrl}
Fix $m\ge 1$. There exists $c_{1},c_{2}>0$ so that if $\ell > c_{1}m+c_{2}$, then we have
\[\left|e^{ik\theta/2}e^{2\pi m \sin\theta}g_{k,m}(e^{i\theta})-2\cos\left(k\theta/2+2\pi m \cos\theta\right)\right|<2\]
for all $\theta\in\prb{\frac{\pi}{2},\frac{2\pi}{3}}$. Furthermore, $c_{1} = \frac{\pi}{2\log\pr{10/7}}$.
\end{prop}
For the proof of Proposition \ref{mrl}, we use the method introduced in the work of W. Duke and P. Jenkins in \cite{dukejen}.

\subsection{Proof of Proposition \ref{mrl}}
We begin with the statement and proof of the following version of Lemma $2$ in \cite{dukejen}:
\begin{lemma}\label{lemma :dukejen lemma}
    Let $R>0$. There exists $A>1$ so that for all $z\in\cF$ with $\abs{j(z)}<R$ we have
    \[g_{k,m}(z)=\int_{-1/2+iA}^{1/2+iA}\frac{\Delta^{\ell}(z)E_{k'}(z)E_{14-k'}(\tau)}{\Delta^{\ell+1}(\tau)\left(j(\tau)-j(z)\right)}e^{2\pi im\tau}d\tau.\]
    where we integrate over the interval $\prb{-\frac{1}{2}+iA,\frac{1}{2}+iA}=\prs{t+iA:t\in\prb{-\frac{1}{2},\frac{1}{2}}}$.
\end{lemma}
\begin{proof}
We have $j(\tau)=q^{-1}+744+P(q)$ where $P(q)=O(\abs{q})$ as $\Im(\tau)\to\infty$. Thus, there exists $N>0$ so that for any $\tau\in\HH$ with $\Im(\tau)>N$, we have $\abs{P(q)}<744$. Choose $A>\max(N,1)$ such that $e^{2\pi A}>R+1488$. For any $\Im(\tau)\ge A$, 
\begin{equation}
    \abs{j(\tau)} \ge \frac{1}{\abs{q}}-744-\abs{P(q)} > e^{2\pi\Im(\tau)}-1488 \ge e^{2\pi A}-1488>R.
\end{equation}
Thus, there exists $A>1$ so that $j\pr{\left[-\frac{1}{2}+iA,\frac{1}{2}+iA\right]} \subset \C \setminus \left\{\abs{z}<R\right\}$. 
Denote $\cC = j\pr{\left[-\frac{1}{2}+iA,\frac{1}{2}+iA\right]}$, and $\gamma:\left[-\frac{1}{2}+iA,\frac{1}{2}+iA\right]\to \C$ as the curve $\gamma(\tau)=j(\tau)$. Then $\gamma$ is a closed, simple, and smooth curve onto $\cC$, oriented clockwise.
Let $U \subset \C$ be the bounded connected component of $\C\setminus\cC$, and let $\abs{\zeta}<R$. So $\zeta\in U$, and the Faber polynomial $F_{k,m}$ is holomorphic on $\overline{U}$. Thus, using Cauchy's formula,
\[F_{k,m}(\zeta)=\frac{1}{2\pi i}\int_{\cC}\frac{F_{k,m}(\xi)}{\xi-\zeta}d\xi = \frac{-1}{2\pi i}\int_{-1/2+iA}^{1/2+iA} \frac{F_{k,m}(j(\tau))}{j(\tau)-\zeta}\frac{dj}{d\tau}(\tau)d\tau.\]
To calculate $\frac{dj}{d\tau}$, we use the following identity, which is due to Ramanujan \cite{ramid}:  \[q\frac{dE_{4}}{dq} = \frac{E_{2}E_{4}-E_{6}}{3}.\]
Hence, 
\[q\frac{dj}{dq} = \frac{3q\frac{dE_{4}}{dq}E_{4}^{2}\Delta -q\frac{d\Delta}{dq} E_{4}^{3}}{\Delta^{2}} = \frac{\left(E_{2}E_{4}-E_{6}\right)E_{4}^{2}\Delta-E_{2}\Delta E_{4}^{3}}{\Delta^2} = -\frac{E_{6}E_{4}^{2}}{\Delta},\]
which, using the chain rule, yields:
\begin{equation}\label{difeq}
    \frac{dj}{d\tau} = \frac{dj}{dq}\frac{dq}{d\tau}=2\pi iq\frac{dj}{dq}=-2\pi i\frac{E_{14}}{\Delta}.
\end{equation}
Therefore, 
\begin{equation}\label{I1}
    F_{k,m}(\zeta) =\int_{-1/2+iA}^{1/2+iA} \frac{F_{k,m}(j(\tau))}{j(\tau)-\zeta}\frac{E_{14}(\tau)}{\Delta(\tau)}d\tau.
\end{equation}
We have $g_{k,m}(\tau) = q^{m}+q^{\ell+1}G(\tau)$, where $G$ is a holomorphic function and $\ord_{\infty}G\ge 0$.
Therefore, we have:
\begin{equation}\label{DicardCoeff}
    F_{k,m}\pr{j(\tau)} =\frac{q^{m}}{\Delta^{\ell}(\tau)E_{k'}(\tau)}+\frac{q^{\ell+1}G(\tau)}{\Delta^{\ell}(\tau)E_{k'}(\tau)}=\frac{q^{m}}{\Delta^{\ell}(\tau)E_{k'}(\tau)}+q\frac{G(\tau)}{P^{\ell}(\tau)E_{k'}(\tau)},
\end{equation}
where $P(\tau)=\prod_{n\ge 1}\pr{1-q^{n}}^{24}$.
Plugging \eqref{DicardCoeff} in \eqref{I1}, we get:
\begin{multline*}
    F_{k,m}\pr{\zeta} = \int_{-1/2+iA}^{1/2+iA}\frac{e^{2\pi im\tau}}{\Delta^{\ell}(\tau)E_{k'}(\tau)}\frac{E_{14}(\tau)}{\Delta(\tau)\pr{j(\tau)-\zeta}}d\tau\\ +\int_{-1/2+iA}^{1/2+iA}q\frac{G(\tau)}{P^{\ell}(\tau)E_{k'}(\tau)}\frac{E_{14}(\tau)}{\Delta(\tau)\pr{j(\tau)-\zeta}}d\tau.
\end{multline*}
Under the change of variables $\tau\to q$, the contour of integration deforms into a circle of radius $e^{-2\pi A}$ around the origin (oriented counter-clockwise).
Doing so with the second integral gives us
\begin{multline*}
    F_{k,m}\pr{\zeta} = \int_{-1/2+iA}^{1/2+iA}\frac{e^{2\pi im\tau}}{\Delta^{\ell}(\tau)E_{k'}(\tau)}\frac{E_{14}(\tau)}{\Delta(\tau)\pr{j(\tau)-\zeta}}d\tau\\ +\frac{1}{2\pi i}\int_{\{\abs{q}=e^{-2\pi A}\}}\frac{G(\tau)}{P^{\ell}(\tau)E_{k'}(\tau)}\frac{E_{14}(\tau)}{\Delta(\tau)\pr{j(\tau)-\zeta}}dq.
\end{multline*}
The function $q\mapsto\frac{G(\tau)}{P^{\ell}(\tau)E_{k'}(\tau)}\frac{F_{k,m}(j(\tau))}{j(\tau)-\zeta}\frac{E_{14}(\tau)}{\Delta(\tau)}$ is holomorphic on the disk $\left\{\abs{q}<e^{-2\pi A}\right\}$, so by Cauchy's Theorem:
\begin{equation}\label{intfola}
    F_{k,m}\pr{\zeta} = \int_{-1/2+iA}^{1/2+iA}\frac{e^{2\pi im\tau}}{\Delta^{\ell+1}(\tau)E_{k'}(\tau)}\frac{E_{14}(\tau)}{j(\tau)-\zeta}d\tau.
\end{equation}
Now, let $z\in\cF$ and suppose that $\abs{j(z)}<R$. Substituting $\zeta=j(z)$ into \eqref{intfola}, using the identity $E_{14-k'}E_{k'}=E_{14}$, and multiplying by $\Delta^\ell(z)E_{k'}(z)$, we get:
\begin{multline}\label{fkmintegral}
    g_{k,m}(z) =\Delta^{\ell}(z)E_{k'}(z)F_{k,m}\pr{j(z)}\\=\int_{-1/2+iA}^{1/2+iA}\frac{\Delta^{\ell}(z)E_{k'}(z)E_{14-k'}(\tau)}{\Delta^{\ell+1}(\tau)\pr{j(\tau)-j(z)}}e^{2\pi im\tau} d\tau. \qedhere
\end{multline}
\end{proof}

Equipped with Lemma \ref{lemma :dukejen lemma}, we will lower the contour of integration, collecting poles as we decrease it from its initial height.

Fix $m\ge1$. For briefness, we denote
\begin{equation}\label{Gdef}
    G(\tau,z) =\frac{\Delta^{\ell}(z)E_{k'}(z)E_{14-k'}(\tau)}{\Delta^{\ell+1}(\tau)\left(j(\tau)-j(z)\right)}e^{2\pi im\tau}.
\end{equation}
The function $\tau\mapsto G(\tau,z)$ has a simple pole whenever $\tau \sim z$ (i.e. there exists $\gamma\in \psl$ such that $z=\gamma.\tau$).
Using \eqref{difeq}, and the identity $E_{14-k'}E_{k'}=E_{14}$ we can write 
\begin{equation}\label{Gres}
    G(\tau,z) = \frac{e^{2\pi im\tau}}{-2\pi i}\frac{\Delta^{\ell}(z)E_{k'}(z)}{\Delta^{\ell}(\tau)E_{k'}(\tau)}\frac{\frac{dj}{d\tau}\pr{\tau}}{j(\tau)-j(z)},
\end{equation}
which will be useful for calculating residues.
\begin{remark}
    The factor $E_{k'}\pr{\tau}$ in the denominator of \eqref{Gres} does not contribute any additional poles, as they are absorbed in the factor of $\frac{dj}{d\tau}\pr{\tau}$ in the numerator, as seen in \eqref{Gdef}.
\end{remark}

\begin{lemma}\label{lemma: possible poles}
    Let $\theta\in\prb{\frac{\pi}{2},\frac{2\pi}{3}}$ and denote $z=e^{i\theta}$. The only possible poles of $\tau\mapsto G\pr{\tau,z}$ in \[D=\left\{\tau\in \HH:\abs{\Re(\tau)}\le\frac{1}{2}, \Im(\tau) > \frac{2}{5}\right\}\] other than $z,-1/z$ are $\frac{-1}{z+1},\frac{z}{z+1}$ and $\frac{-1}{z-1},\frac{-z}{z-1}$.
\end{lemma}
 In other words, for any $\tau\sim z$ such that $\tau\notin\prs{ z,-1/z, \frac{-1}{z+1},\frac{z}{z+1},\frac{-1}{z-1},\frac{-z}{z-1}}$, we have $\tau\notin D$.
 
\begin{proof}
Let $\tau\in \HH$ such that $\tau\sim z$ and $\tau\notin\prs{ z,-1/z, \frac{-1}{z+1},\frac{z}{z+1},\frac{-1}{z-1},\frac{-z}{z-1}}$. Hence, there exists $\gamma=\begin{psmallmatrix}
    a & b\\
    c & d
\end{psmallmatrix} \in \psl$ such that $\tau = \gamma.z$ and  \[\gamma\notin\prs{\text{id},\begin{psmallmatrix}
    0 & -1\\
    1 & 0
\end{psmallmatrix},\begin{psmallmatrix}
    0 & -1\\
    1 & 1
\end{psmallmatrix},\begin{psmallmatrix}
    1 & 0\\
    1 & 1
\end{psmallmatrix},\begin{psmallmatrix}
    0 & -1\\
    1 & -1
\end{psmallmatrix},\begin{psmallmatrix}
    -1 & 0\\
    1 & -1
\end{psmallmatrix}}.\]
We divide into cases:
\begin{enumerate}[label = Case \Roman*.]
    \item Suppose $\abs{cd}=0$:\\
    If $c=0$, then $\gamma=\begin{psmallmatrix}
    1 & b\\
    0 & 1
    \end{psmallmatrix}$. Now, if $z=\rho$ then since $\gamma\neq\text{id}$ and $\tau\neq-1/z=\rho+1$ we must have $b\neq 0,1$ and therefore $\abs{\Re(\gamma.z)}>\frac{1}{2}$. If $z\neq \rho$, since $\gamma\neq\text{id}$ we must have $\abs{\Re(\gamma.z)}>\frac{1}{2}$.\\
    Otherwise, $d=0$ and therefore $-bc=1$. WLOG assume that $c=1$ so $\gamma=\begin{psmallmatrix}
    a & -1\\
    1 & 0
    \end{psmallmatrix}$. We have $\gamma.z = a-1/z$, and thus $\Re\pr{\gamma.z} = a-\cos\theta$. If $z=\rho$, since $\tau\neq\rho=-1-\frac{1}{\rho}$ and $\gamma\neq \begin{psmallmatrix}
        0 & -1\\
        1 & 0
    \end{psmallmatrix}$, we must have $a\neq -1,0$ and therefore $\abs{\Re\pr{\gamma.z}}>\frac{1}{2}$.  If $z\neq \rho$ since $\gamma\neq \begin{psmallmatrix}
        0 & -1\\
        1 & 0
    \end{psmallmatrix}$ we have $a\neq 0$ and thus $\abs{\Re\pr{\gamma.z}}>\frac{1}{2}$. In any case, $\tau=\gamma.z\notin D$.
    \item Suppose $\abs{cd}=1$:\\
    Since $\gamma\in \psl$, we can assume WLOG that $c=1$.\\
    If $d=1$, then $a-b=1$, and thus $\gamma.z = a+\frac{-1}{z+1}$. Therefore $\Re(\gamma.z) = a -\frac{1}{2}$ but since $\gamma \neq \begin{psmallmatrix}
    0 & -1\\
    1 & 1
\end{psmallmatrix},\begin{psmallmatrix}
    1 & 0\\
    1 & 1
\end{psmallmatrix}$, we have $a\neq 0,1$ and thus $\abs{\Re(\gamma.z)}>\frac{1}{2}$.\\
    If $d=-1$, then $a+b=-1$, and thus $\gamma.z = a+\frac{-1}{z-1}$. Therefore $\Re(\gamma.z) = a +\frac{1}{2}$ but since $\gamma \neq \begin{psmallmatrix}
    0 & -1\\
    1 & -1
\end{psmallmatrix},\begin{psmallmatrix}
    -1 & 0\\
    1 & -1
\end{psmallmatrix}$, we have $a\neq 0,-1$ and thus $\abs{\Re(\gamma.z)}>\frac{1}{2}$.
    Hence, $\tau=\gamma.z\notin D$.
    \item Suppose $\abs{cd}\ge 2$:\\
    In this case, $c^{2}+d^{2}\ge 5$, and \[\abs{ce^{i\theta}+d}^{2}=c^{2}+d^{2}+2cd\cos\theta\ge c^{2}+d^{2} -\abs{cd}\ge \frac{1}{2}\pr{c^{2}+d^{2}} \ge \frac{5}{2}.  \]
    Hence, \[\Im(\gamma.e^{i\theta})=\frac{\Im(e^{i\theta})}{\abs{ce^{i\theta}+d}^{2}} \le  \frac{2}{5}\] and therefore $\tau=\gamma.z\notin D$.\qedhere
\end{enumerate}
\end{proof}

Assume $z=e^{i\theta}$, where $\theta\in\pr{\frac{\pi}{2},\frac{2\pi}{3}}$. Since $j(e^{i\theta})\in \prb{0,1728}$, we can choose $R > 1728$ and by Lemma \ref{lemma :dukejen lemma} get an $A>1$ such that: 
\[g_{k,m}(z) = \int_{-1/2+iA}^{1/2+iA}G(\tau,z)d\tau.\]
Using Cauchy's residue theorem, we begin to lower the contour of integration from its initial height $A$, collecting poles along the way. Fix some $\frac{\sqrt{3}}{2}<A'<\sin\theta$, integrating along a rectangular contour as demonstrated in Figure \ref{fig:First decrease}, we decrease our contour from height $A$ to height $A'$.
\begin{figure}[ht]
    \begin{center}
    \begin{tikzpicture}[scale = 2] 
    \filldraw (0,0) circle (0.5pt) node[below left] {$0$};
    \filldraw (1,0) circle (0.5pt) node[below right] {$1$};
    \filldraw (-1,0) circle (0.5pt) node[below left] {$-1$};
    \filldraw (-0.227,0.974) circle (0.5pt) node[above left] {$z$};
    \filldraw (0.227,0.974) circle (0.5pt) node[above] {$\nicefrac{-1}{z}$};
    \draw[thick,->] (-1.25,0) -- (1.25,0) node[right] {};
    \draw[line width=0.15mm,dashed] (1,0) arc (0:60:1);
    \draw[line width=0.2mm] (0.5,0.866) arc (60:120:1);
    \draw[line width=0.15mm,dashed](-0.5,0.866) arc (120:180:1);
    \draw[line width=0.2mm] (0.5,0.866) -- (0.5,2.6);
    \draw[line width=0.2mm] (-0.5,0.866) -- (-0.5,2.6);
    \draw[dashed] (0.5,0) -- (0.5,.866);
    \draw[dashed] (-0.5,0) -- (-0.5,.866);
    \draw[thick] (-0.5,2.3) to node[currarrow,rotate = 90] {} (-0.5,.92) to node[currarrow,rotate = 180]{}(0.5,.92) to node[currarrow,rotate=-90]{} (0.5,2.3)to node[currarrow]{}(-0.5,2.3);
    \filldraw (0.5,2.3) circle (0pt) node[right] {$A$};
    \filldraw (0.5,0.92) circle (0pt) node[right] {$A'$};
    \end{tikzpicture}
    \caption{The contour of integration}
        \label{fig:First decrease}
    \end{center}
\end{figure}

We have two poles inside the region of integration at $\tau=z$  and $\tau=-1/z$. Calculating the residues using \eqref{Gres}, we obtain:
\begin{equation*}
    \res_{\tau = z}G(\tau,z) =\frac{e^{2\pi imz}}{-2\pi i},\quad\res_{\tau = -1/z}G(\tau,z) = \frac{z^{-k}e^{2\pi im(-1/z)}}{-2\pi i}.
\end{equation*}
Since $\tau\mapsto G\pr{\tau,z}$ is $1$-periodic and the two vertical parts of the integral are of opposite orientation, their contribution cancels pairwise. Therefore, we get:
\[\int_{-1/2+iA'}^{1/2+iA'}G(\tau,e^{i\theta})d\tau =g_{k,m}(e^{i\theta})-e^{2\pi ime^{i\theta}}-e^{-ik\theta}e^{-2\pi ime^{-i\theta}}.\]
Multiplying both sides by $e^{ik\theta/2}e^{2\pi m\sin\theta}$ gives us:
\begin{multline}\label{int-differ relation}
    e^{ik\theta/2}e^{2\pi m\sin\theta}\int_{-1/2+iA'}^{1/2+iA'}G(\tau,e^{i\theta})d\tau\\=e^{ik\theta/2}e^{2\pi m \sin\theta}g_{k,m}(e^{i\theta})-2\cos\left(k\theta/2+2\pi m \cos\theta\right).
\end{multline}
For convenience, we denote:
\begin{equation}
    I(a) = e^{2\pi m\sin\theta}\int_{-1/2+ia}^{1/2+ia}\left|G(\tau,e^{i\theta})\right|d\tau.
\end{equation}

We see that
 \begin{equation}\label{heightsofpoles}
     \Im\left(\frac{-1}{e^{i\theta}+1}\right)=\Im\left(\frac{e^{i\theta}}{e^{i\theta}+1}\right) = \frac{\sin\theta}{2+2\cos\theta}\ge \frac{\sin\theta}{2-2\cos\theta}=\Im\left(\frac{-e^{i\theta}}{e^{i\theta}-1}\right)
 \end{equation}
and $\Re\left(\frac{-1}{e^{i\theta}+1}\right)=\Re\left(\frac{-e^{i\theta}}{e^{i\theta}-1}\right) = -\frac{1}{2}$.  The function $\theta\mapsto \frac{\sin\theta}{2+2\cos\theta}$ is increasing on the interval $\pr{\frac{\pi}{2},\frac{2\pi}{3}}$ from $\frac{1}{2}$ to $\frac{\sqrt{3}}{2}$, and the function $\theta\mapsto \frac{\sin\theta}{2-2\cos\theta}$ is decreasing on the interval $\pr{\frac{\pi}{2},\frac{2\pi}{3}}$ from $\frac{1}{2}$ to $\frac{\sqrt{3}}{6}$. Our goal now is to decrease the path of integration from its current height $A'$ to some $\frac{1}{2}<A''<\frac{\sqrt{3}}{2}$ to avoid the contribution of the poles $\frac{-1}{z-1},\frac{-z}{z-1}$. However, for values of $\theta$ sufficiently close to $\frac{2\pi}{3}$, we cannot decrease our path from its current height without accounting for the contribution of the poles $\frac{-1}{z+1},\frac{z}{z+1}$. Hence, we divide into cases:
\subsubsection{Suppose $\frac{\pi}{2}<\theta<1.9$}
We have $\Im\left(\frac{-1}{1+e^{i\theta}}\right)<0.75<\frac{\sqrt{3}}{2}$, so we set $A''=0.75$ and decrease the contour of integration from height $A'$ to $A''$ as before, integrating along a rectangular contour, as demonstrated in Figure \ref{fig: contour pi/2}.
\begin{figure}[ht]
    \begin{center}
    \begin{tikzpicture}[scale = 3]
    \filldraw (0,0) circle (0.4pt) node[below left] {$0$};
    \filldraw (1,0) circle (0.4pt) node[below right] {$1$};
    \filldraw (-1,0) circle (0.4pt) node[below left] {$-1$};
    \filldraw (-0.227,0.974) circle (0.4pt) node[above left] {$z$};
    \filldraw (0.227,0.974) circle (0.4pt) node[above] {$\nicefrac{-1}{z}$};
    \draw[thick,->] (-1.25,0) -- (1.25,0) node[right] {};
    \draw[line width=0.15mm, loosely dashed] (1,0) arc (0:60:1);
    \draw[line width=0.2mm] (0.5,0.866) arc (60:120:1);
    \draw[line width=0.15mm, loosely dashed](-0.5,0.866) arc (120:180:1);
    \draw[line width=0.2mm] (0.5,0.866) -- (0.5,1.21);
    \draw[line width=0.2mm] (-0.5,0.866) -- (-0.5,1.21);
    \draw[loosely dashed] (0.5,0) -- (0.5,.866);
    \draw[loosely dashed] (-0.5,0) -- (-0.5,.866);
    \draw[thick] (-0.5,0.92) to node[currarrow,rotate = 90] {}  (-0.5, 0.75) to node[currarrow,rotate = 180]{}(0.5,.75) to node[currarrow,rotate=-90]{} (0.5,0.92)to node[currarrow]{}(-0.5,0.92);
    \filldraw[fill = white, draw = white] (0.525, 0.69) rectangle (0.7,0.81);
    \filldraw (0.5,0.75) circle (0pt) node[right] {$A''$};
    \filldraw (0.5,0.92) circle (0pt) node[right] {$A'$};
    
    \end{tikzpicture}
    \caption{The contour of integration when $\theta\in\left(\frac{\pi}{2},1.9\right)$.}
    \label{fig: contour pi/2}
    \end{center}
\end{figure}
As before, the contribution of the two vertical parts of the integral cancels pairwise.
Hence,
\[\int_{-1/2+iA'}^{1/2+iA'}G(\tau,e^{i\theta})d\tau = \int_{-1/2+0.75i}^{1/2+0.75i}G(\tau,e^{i\theta})d\tau.\]
Substituting in \eqref{int-differ relation}, we obtain
\begin{multline}\label{eq: int-differ relation 075}
    e^{ik\theta/2}e^{2\pi m\sin\theta}\int_{-1/2+0.75i}^{1/2+0.75i}G(\tau,e^{i\theta})d\tau\\=e^{ik\theta/2}e^{2\pi m \sin\theta}g_{m}(e^{i\theta})-2\cos\left(k\theta/2+2\pi m \cos\theta\right).
\end{multline}
Thus, the quantity we need to bound is:
\begin{equation}\label{int75}
I(0.75) \le \max_{|x|\le1/2}e^{2\pi m(\sin\theta-0.75)}\left|\frac{\Delta(e^{i\theta})}{\Delta(x+0.75i)}\right|^{\ell}\left|\frac{E_{k'}(e^{i\theta})E_{14-k'}(x+0.75i)}{\Delta(x+0.75i)\left(j(x+0.75i)-j(e^{i\theta})\right)}\right|.
\end{equation}
Lemma \ref{lemma: possible poles} and \eqref{heightsofpoles} show that for any $x\in\prb{-\frac{1}{2},\frac{1}{2}}$ and any $\theta\in \prb{\frac{\pi}{2},1.9}$ we have $x+0.75i\not\sim e^{i\theta}$ and thus $j(x+0.75i)\neq j(e^{i\theta})$. Since $j$ is holomorphic, the function $\pr{x,\theta}\mapsto j(x+0.75i)-j(e^{i\theta})$ is continuous and non-zero on $\prb{-\frac{1}{2},\frac{1}{2}}\times\prb{\frac{\pi}{2},1.9}$. In addition, $\Delta(z)\neq0$ for all $z\in\HH$ and $E_{k'},\,E_{14-k'},\,\Delta$, are all holomorphic in $\HH$. Therefore, the function \[\pr{x,\theta}\mapsto \left|\frac{E_{k'}(e^{i\theta})E_{14-k'}(x+0.75i)}{\Delta(x+0.75i)\left(j(x+0.75i)-j(e^{i\theta})\right)}\right|,\] is well defined and continuous on $\prb{-\frac{1}{2},\frac{1}{2}}\times\prb{\frac{\pi}{2},1.9}$. Since $\prb{-\frac{1}{2},\frac{1}{2}}\times\prb{\frac{\pi}{2},1.9}$ is compact, there exists $B_{1}\ge 0$ such that for all $\theta\in\prb{\frac{\pi}{2},1.9}$
 \begin{equation}\label{B75}
     \max_{\abs{x}\le\frac{1}{2}}\left|\frac{E_{k'}(e^{i\theta})E_{14-k'}(x+0.75i)}{\Delta(x+0.75i)\left(j(x+0.75i)-j(e^{i\theta})\right)}\right|<e^{B_{1}}.
 \end{equation}
Additionally, by $\eqref{deltabounds75}$, for all  $x\in \left[-\frac{1}{2},\frac{1}{2}\right]$ and all $\theta\in\left(\frac{\pi}{2},\frac{2\pi}{3}\right)$,
\begin{equation}\label{C75}
    \abs{\frac{\Delta(e^{i\theta})}{\Delta(x+0.75i)}} < e^{-\log\pr{10/7}}
.\end{equation}
Plugging \eqref{B75} and \eqref{C75} into \eqref{int75}, we obtain:
\begin{equation}\label{exp bnd 75}
    I(0.75) \le e^{2\pi m(\sin\theta-0.75)-\log\pr{10/7}\ell+B_{1}}.
\end{equation}
\subsubsection{Suppose $1.9 \le \theta < \frac{2\pi}{3}$}
In this case, the pole at $\tau = \frac{-1}{e^{i\theta}-1}$ is of height at most $0.4$ and $\tau = \frac{-1}{e^{i\theta}+1}$ is of height at least $0.69$. Hence, we will decrease $A'$ to $A''=0.65$. In this case, we slightly change our rectangular contour, adding a little semicircle to our vertical parts, so that one of the poles $\frac{-1}{z+1}$, $\frac{z}{z+1}$ is inside the region of integration, as demonstrated in Figure \ref{fig: contour 2pi/3}.

\begin{figure}[ht]
    \begin{center}
    \begin{tikzpicture}[scale = 3]
    \filldraw (0,0) circle (0.4pt) node[below left] {$0$};
    \filldraw (1,0) circle (0.4pt) node[below right] {$1$};
    \filldraw (-1,0) circle (0.4pt) node[below left] {$-1$};
    \filldraw (-0.3327,0.943) circle (0.4pt) node[above] {$z$};
    \filldraw (0.3327,0.943) circle (0.4pt) node[above] {$\nicefrac{-1}{z}\,\,\,$};
    \filldraw (-0.5,0.737) circle (0.4pt) node[right] {\footnotesize $\nicefrac{-1}{(1+z)}$};
    \draw[thick,->] (-1.25,0) -- (1.25,0) node[right] {};
    \draw[line width=0.15mm,loosely dashed] (1,0) arc (0:60:1);
    \draw[line width=0.2mm] (0.5,0.866) arc (60:120:1);
    \draw[line width=0.15mm,loosely dashed](-0.5,0.866) arc (120:180:1);
    \draw[line width=0.2mm] (0.5,0.866) -- (0.5,1.21);
    \draw[line width=0.2mm] (-0.5,0.866) -- (-0.5,1.21);
    \draw[loosely dashed] (0.5,0) -- (0.5,.866);
    \draw[loosely dashed] (-0.5,0) -- (-0.5,.866);
    \draw[thick] (-0.5,0.9) to node[currarrow,rotate = 90] {} (-0.5,.787) arc (90:270:0.05) to (-0.5, 0.65) to node[currarrow,rotate = 180]{}(0.5,.65) to (0.5,.687) arc (270:90:.05) to node[currarrow,rotate=-90]{} (0.5,0.9)to node[currarrow]{}(-0.5,0.9);
    \filldraw[fill = white, draw = white] (0.525, 0.67) rectangle (0.85,0.82);
    \filldraw (0.5,0.737) circle (0.4pt) node[right] {\footnotesize $\nicefrac{z}{(1+z)}$};
    \filldraw (0.5,0.61) circle (0pt) node[right] {$A''$};
    \filldraw (0.5,0.91) circle (0pt) node[right] {$A'$};
    \end{tikzpicture}
    \caption{The contour of integration when $\theta\in\left[1.9,\frac{2\pi}{3}\right)$.}
    \label{fig: contour 2pi/3}
    \end{center}
\end{figure}

As before, the contribution of the two vertical parts of the integral cancels pairwise.
Hence,
\begin{multline}\label{withpole}
    \int_{-1/2+iA'}^{1/2+iA'}G(\tau,e^{i\theta})d\tau \\
    = e^{ik\theta/2}e^{2\pi m\sin\theta}\left(-2\pi i\Res_{\tau = \frac{-1}{1+z}}G(\tau,z) + \int_{-1/2+0.65i}^{1/2+0.65i}G(\tau,e^{i\theta})d\tau\right).
\end{multline}
Calculating the residue, we get:
\begin{align*}
    \res_{\tau = \frac{-1}{1+z}}G(\tau,z) = \frac{(z+1)^{-k}e^{2\pi im\left(\frac{-1}{1+z}\right)}}{-2\pi i}.
\end{align*}
Therefore,
\begin{equation*}
    -2\pi ie^{ik\theta/2}e^{2\pi m\sin\theta}\Res_{\tau = \frac{-1}{1+z}}G(\tau,z) = \frac{e^{\pi m\left(2\sin\theta-\tan(\theta/2)\right)}e^{-\pi  im}}{\left(2\cos(\theta/2)\right)^{k}}
.\end{equation*}
Together with \eqref{int-differ relation} and \eqref{withpole}, we have:
\begin{equation}\label{Bound of 065}
    \left|e^{ik\theta/2}e^{2\pi m \sin\theta}g_{m}(e^{i\theta})-2\cos\left(k\theta/2+2\pi m \cos\theta\right)\right|\le \frac{e^{\pi m\left(2\sin\theta-\tan(\theta/2)\right)}}{\left(2\cos(\theta/2)\right)^{k}} +I(0.65).
\end{equation}

Consider the derivative (with respect to $\theta$) of the first term on the right-hand side of \eqref{Bound of 065}:
\begin{equation*}
    \left(\frac{e^{\pi m\left(2\sin\theta-\tan(\theta/2)\right)}}{\left(2\cos(\theta/2)\right)^{k}} \right)'\\ = \frac{e^{\pi m\left(2\sin\theta-\tan(\theta/2)\right)}\left((k\sin\theta)/2+\pi m\left(4\cos^{2}(\theta/2)\cos\theta - 1\right)\right)}{2^{k+1}\cos^{k+2}(\theta/2)}.
\end{equation*}
Therefore, if $k\ge \frac{8\pi}{\sqrt{3}}m$, then for all $\theta \in\left[\frac{\pi}{2},\frac{2\pi}{3}\right]$, we get:
\begin{equation*}
    \left(\frac{e^{\pi m\left(2\sin\theta-\tan(\theta/2)\right)}}{\left(2\cos(\theta/2)\right)^{k}} \right)' \ge \frac{e^{\pi m\left(2\sin\theta-\tan(\theta/2)\right)}\left(\frac{\sqrt{3}k}{4}+\pi m\left(-1 - 1\right)\right)}{2^{k+1}\cos^{k+2}(\theta/2)}\ge 0
.\end{equation*}
This shows that the first term on the right-hand side of \eqref{Bound of 065} is increasing, so it is bounded by its value at $\frac{2\pi}{3}$, which means: 
\begin{equation}\label{eq: Pole is bound 065}
    \frac{e^{\pi m\left(2\sin\theta-\tan(\theta/2)\right)}}{\left(2\cos(\theta/2)\right)^{k}}\le 1
\end{equation}
for all $\theta\in\left[\frac{\pi}{2},\frac{2\pi}{3}\right]$. Thus, we are left to deal with 
\begin{equation}\label{int65}
    I(0.65) \le \max_{|x|\le1/2}e^{2\pi m(\sin\theta-0.65)}\left|\frac{\Delta(e^{i\theta})}{\Delta(x+0.65i)}\right|^{\ell}\left|\frac{E_{k'}(e^{i\theta})E_{14-k'}(x+0.65i)}{\Delta(x+0.65i)\left(j(x+0.65i)-j(e^{i\theta})\right)}\right|.
\end{equation} 
Lemma \ref{lemma: possible poles} and \eqref{heightsofpoles} show that $x+0.65i\not\sim e^{i\theta}$ for all $x\in\prb{-\frac{1}{2},\frac{1}{2}}$ and any $\theta\in\prb{1.9,\frac{2\pi}{3}}$. As in the previous case, this is sufficient to imply that there exists $B_{2}\ge 0$ such that for all $\theta\in\prb{1.9,\frac{2\pi}{3}}$
 \begin{equation}\label{B65}
     \max_{\abs{x}\le\frac{1}{2}}\left|\frac{E_{k'}(e^{i\theta})E_{14-k'}(x+0.65i)}{\Delta(x+0.65i)\left(j(x+0.65i)-j(e^{i\theta})\right)}\right|<e^{B_{2}}.
 \end{equation}
Additionally, by $\eqref{deltabounds65}$, for all $x\in\left[-\frac{1}{2},\frac{1}{2}\right]$ and all $\theta\in\left(\frac{\pi}{2},\frac{2\pi}{3}\right)$,
\begin{equation}\label{C65}
    \abs{\frac{\Delta(e^{i\theta})}{\Delta(x+0.65i)}} < e^{-\log\pr{2}}
.\end{equation}
Plugging \eqref{B65} and \eqref{C65} into \eqref{int65}, we obtain:
\begin{equation}\label{exp bnd 65}
    I(0.65) \le e^{2\pi m(\sin\theta-0.65)-\log\pr{2}\ell+B_{2}}.
\end{equation}
\subsubsection{Back to $\theta\in\pr{\frac{\pi}{2},\frac{2\pi}{3}}$}
Denote 
\begin{equation}\label{c_{1}}
    c_{1} =\max\pr{\frac{\pi}{2\log\pr{10/7}},\frac{7\pi}{10\log\pr{2}}}=\frac{\pi}{2\log\pr{10/7}},
\end{equation} and 
\begin{equation}\label{c_{2}}
    c_{2}=\max\left(\frac{B_{1}-\log\pr{1.995}}{\log\pr{10/7}},\frac{B_{2}-\log\pr{0.995}}{\log\pr{2}}\right).
\end{equation}
Suppose $\ell> c_{1} m+c_{2}$, so by \eqref{exp bnd 75} 
\begin{equation*}
    I(0.75)\le e^{\frac{\pi m}{2} -\log\pr{10/7}\ell +B_{1}} < e^{\log\pr{1.995}}=1.995.
\end{equation*}
By \eqref{exp bnd 65}, 
\begin{equation*}
    1+I(0.65) \le 1+e^{\frac{7\pi m}{10} -\log\pr{2}\ell +B_{2}} < 1+e^{\log\pr{0.995}}=1.995.
\end{equation*}
Therefore,
\begin{equation*}
    \max\left(I(0.65)+1,I(0.75)\right)< 1.995.
\end{equation*}
Finally, using \eqref{eq: int-differ relation 075}, \eqref{Bound of 065} and \eqref{eq: Pole is bound 065}, for all $\theta\in\pr{\frac{\pi}{2},\frac{2\pi}{3}}$ we get:
\begin{equation*}
    \abs{e^{ik\theta/2}e^{2\pi m \sin\theta}g_{m}(e^{i\theta})-2\cos\left(k\theta/2+2\pi m \cos\theta\right)}\le \max\left(I(0.65)+1,I(0.75)\right)
,\end{equation*}
which implies:
\[\abs{e^{ik\theta/2}e^{2\pi m \sin\theta}g_{k,m}(e^{i\theta})-2\cos\left(k\theta/2+2\pi m \cos\theta\right)}<1.995.\]
Using the continuity of left-hand side of the expression above, for any $\theta\in\prb{\frac{\pi}{2},\frac{2\pi}{3}}$ we have
\begin{equation}\label{MR Eq}
    \abs{e^{ik\theta/2}e^{2\pi m \sin\theta}g_{k,m}(e^{i\theta})-2\cos\left(k\theta/2+2\pi m \cos\theta\right)}\le 1.995<2,
\end{equation}
which completes our proof of Proposition \ref{mrl}. 
\subsection{Proof of Theorem \ref{mr}}
Choose $\alpha = 4.5 \ge c_{1}$ and $\beta = c_{2}$. Let $k\ge0$ be an even integer with $\ell>\alpha m +\beta$. Denote $D=\ell-m$ and $h(\theta)= \frac{k\theta}{2}+2\pi m \cos\theta$. Consider the derivative of $h$: \[h'(\theta) = k/2-2\pi m \sin\theta= \left(k-4\pi m\sin\theta\right)/2\ge\left(k-4\pi m\right)/2. \]
Since $k\ge 12\ell > 12\alpha m \ge 4\pi m$, we have $h'(\theta)\ge 0$. Therefore, $h$ increases from $\frac{\pi k}{4}=3\pi\ell+\frac{\pi k'}{4}$ to $\frac{\pi k}{3}-\pi m = 3\pi\ell+\frac{\pi k'}{3}+\pi D$, passing through $D+1$ consecutive integer multiples of $\pi$. Let $n_0$ be the least integer such that $\pi n_{0}\in\left[\frac{\pi k}{4},\frac{\pi k}{3}-\pi m\right]$. Thus, there exists $\frac{\pi}{2}\le \theta_{0}<\theta_{1}<\ldots<\theta_{D}\le \frac{2\pi}{3}$ such that $h\pr{\theta_{j}}=\pi n_{0} +\pi j$ for all $0\le j\le D$. We have $\cos\left(h\pr{\theta_{j}}\right)=(-1)^{j+n_{0}}$. So, by Proposition \ref{mrl} we get:
\[2+2\cdot(-1)^{j+n_{0}}>e^{ik\theta_{j}/2}e^{2\pi m \sin\theta_{j}}g_{k,m}(e^{i\theta_{j}})>-2+2\cdot(-1)^{j+n_{0}}\]
Hence, the function $\theta\mapsto e^{ik\theta/2}e^{2\pi m \sin\theta}g_{k,m}(e^{i\theta})$ is continuous and changes its sign in the interval $\left(\theta_{j-1},\theta_{j}\right)$, showing that it attains the value zero at least once in each of those intervals. Since there are $D$ intervals, we deduce that $g_{k,m}(e^{i\theta})$ has $D$ zeros. Hence, the form $g_{k,m}$ has $D=\ell -m$ zeros on the arc $\cA$, in addition to the $m$ zeros of $g_{k,m}$ at infinity. In other words, all of the finite zeros of $g_{k,m}$ in the fundamental domain are on the arc $\cA$. We are left to show that the zeros of $g_{k,m}$ are uniformly distributed on the arc:\\ Let $z_{1},\ldots,z_{D}\in\prb{\pi/2,2\pi/3}$ be the $D$ zeros of the form $g_{k,m}$ under the parametrization $\theta\mapsto e^{i\theta}$ and let $\prb{a,b}\sub \prb{\pi/2,2\pi/3}$. Since $h$ is increasing, we have $z_{j}\in\prb{a,b}$ if and only if $h\pr{z_j}\in \prb{h\pr{a},h\pr{b}}$. By definition $h\pr{z_j}$ is between two consecutive integer multiples of $\pi$, therefore the number of zeros in the interval $\prb{a,b}$ is roughly the number of integer multiples in the interval $\prb{h\pr{a},h\pr{b}}$, i.e.\
\begin{multline*}
    \#\prs{1\le j\le D:z_{j}\in\prb{a,b}}\\ = \#\prs{n\in\Z:\pi n\in \prb{h\pr{a},h\pr{b}}}+O\pr{1}
    =\floor{\frac{h\pr{b}-h\pr{a}}{\pi}}+O\pr{1}.    
\end{multline*}

Hence,
\begin{multline*}
    \frac{\#\prs{1\le j\le D:z_{j}\in\prb{a,b}}}{D} = \frac{h\pr{b}-h\pr{a}}{\pi D}+O\pr{\frac{1}{D}} \\
    = \frac{k\pr{b-a}}{2\pi D}+2\pi m\frac{\cos\pr{b}-\cos\pr{a}}{D} + O\pr{\frac{1}{D}} \\=\frac{\pr{12\ell+k'}\pr{b-a}}{2\pi \pr{\ell-m}}+O\pr{\frac{1}{D}} \xrightarrow[]{\ell\to\infty}\frac{6\pr{b-a}}{\pi}=\frac{b-a}{\frac{2\pi}{3}-\frac{\pi}{2}}.
\end{multline*}
Thus, the zeros become uniformly distributed which implies Theorem \ref{mr}.

\section{Quantifying the Bounds}\label{What are the bounds}
In this section, we will quantify the constants $c_{1}$ and $c_{2}$ in Proposition \ref{mrl}, which will determine $\alpha$ and $\beta$ in Theorem \ref{mr}. Specifically, we will prove:
\begin{thm}\label{QuanMR}
    Fix $m\ge 1$. For all $k=12\ell+k'$, if $\ell > 4.5m+9.5$ then all the zeros of $g_{k,m}$ in the fundamental domain lie on the arc $\left\{e^{i\theta}:\frac{\pi}{2}\le \theta\le \frac{2\pi}{3}\right\}$.
\end{thm}
\subsection{Outline of the Proof of Theorem \ref{QuanMR}}
As shown in Proposition \ref{mrl}, we have $c_{1}=\frac{\pi}{2\log(10/7)}\le 4.5$. Thus, we are left to quantify $c_{2}$.
To determine $c_{2}$, we need to quantify the bounds $B_{1}$ in \eqref{B75}
 \begin{equation}
     \max_{\abs{x}\le\frac{1}{2}}\left|\frac{E_{k'}(e^{i\theta})E_{14-k'}(x+0.75i)}{\Delta(x+0.75i)\left(j(x+0.75i)-j(e^{i\theta})\right)}\right|<e^{B_{1}}, \tag{\ref{B75}}
 \end{equation}
and $B_{2}$ in \eqref{B65}
 \begin{equation*}
     \max_{\abs{x}\le\frac{1}{2}}\left|\frac{E_{k'}(e^{i\theta})E_{14-k'}(x+0.65i)}{\Delta(x+0.65i)\left(j(x+0.65i)-j(e^{i\theta})\right)}\right|<e^{B_{2}}. \tag{\ref{B65}}
 \end{equation*}
 We will do so in the following manner:
 \begin{enumerate}
     \item In \S\ref{j bounds}, we will find a lower bound to the difference
     \[\min_{\abs{x}\le \frac{1}{2}}\abs{j\pr{x+0.75i}-j\pr{e^{i\theta}}}\]
     where $\theta\in[\pi/2,1.9]$, and to
     \[\min_{\abs{x}\le \frac{1}{2}}\abs{j\pr{x+0.65i}-j\pr{e^{i\theta}}}\]
     where $\theta\in[1.9,2\pi/3]$.
     \item In \S\ref{E_k monotinicity}, \S\ref{E_4 bounds} and \S\ref{E_6 bounds}, we will find upper bounds for $E_{4}$ and $E_{6}$ on the arc $\cA$, and on the horizontal lines of height $0.65$ and $0.75$. Those bounds, together with the identities, 
     \begin{align*}
         E_{8} & = E_{4}^2 ,\\
         E_{10} & = E_{4}E_{6}, \\
         E_{14} & = E_{4}^{2}E_{6}, 
     \end{align*}
    would yield upper bounds for $E_{k'}$ for all $k'\in \{0,4,6,8,10,14\}$.
     \item We will use Proposition \ref{tau bounds} to obtain a lower bound on $\Delta(x+0.65i)$ and $\Delta(x+0.75i)$.
 \end{enumerate}
Lastly, we will substitute the bounds above in \eqref{B75} and \eqref{B65} and get $B_{1}=3.94$ and $B_{2}=5.12$.
\subsection{Bounds for the $j$-function}\label{j bounds} 
We begin with a few basic facts on the $j$-function:
\begin{enumerate}
    \item The $q$-expansion of $j$ has integer coefficients, thus $\overline{j(x+iy)} = j(-x +iy)$.
    \item  The $j$-function is injective on the fundamental domain and onto $\C$.
    \item  $j(\tau)$ is real if and only if $\tau$ lies on the boundary of the fundamental domain or the imaginary line. Specifically, $j$ maps the line $\left\{-\frac{1}{2}+it:t>\frac{\sqrt{3}}{2}\right\}$ onto $(-\infty,0)$, the arc $\cA=\left\{e^{i\theta}:\frac{\pi}{2}\le \theta\le \frac{2\pi}{3}\right\}$ onto the interval $[0,1728]$, and the line $\left\{it:t>1\right\}$ onto $(1728,\infty)$.
    \item For all $n\ge 1$ we have $1\le c\pr{n}\le \frac{e^{4\pi\sqrt{n}}}{\sqrt{2}n^{3/4}}$, where $c\pr{n}$ are the Fourier coefficients of $j$. This was proven in \cite{jbounds}. 
\end{enumerate}

Our goal in this section is to prove the following proposition:
\begin{prop}\label{lemma:j bounds}
\begin{enumerate}[label = (\roman*)]
    \item For all $\theta\in\left[\frac{\pi}{2},1.9\right]$ we have
    \begin{equation}\label{eq: j-diff 75}
        \min_{x\in \left[-0.5,0.5\right]}\abs{j(x+0.75i)-j(e^{i\theta})}\ge 176.
    \end{equation}
    \item For all $\theta\in\left[1.9,\frac{2\pi}{3}\right]$ we have
    \begin{equation}\label{eq: j-diff 65}
        \min_{x\in \left[-0.5,0.5\right]}\abs{j(x+0.65i)-j(e^{i\theta})}\ge 311.
    \end{equation}
\end{enumerate}
\end{prop}

Plotting the function $x\mapsto\abs{j\pr{x+0.75i}-j\pr{e^{i\theta}}}$ for various values of $\theta\in \prb{\frac{\pi}{2},1.9}$, see Figure \ref{fig:various values}, we can see that this function is decreasing on $\prb{0,0.5}$ and attains its minimum at $x=\frac{1}{2}$.
\begin{figure}[ht]
    \centering
    \includegraphics[height=65mm]{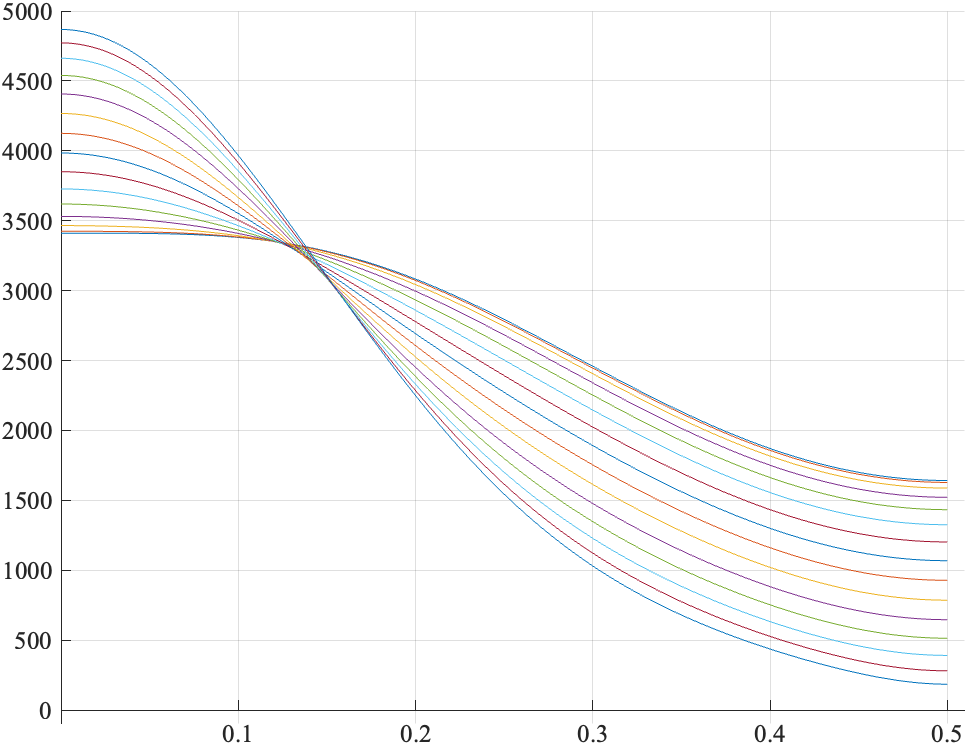}
    \caption{The function $x\mapsto\abs{j\pr{x+0.75i}-j\pr{e^{i\theta}}}$ for various values of $\theta\in \prb{\frac{\pi}{2},1.9}$.}
    \label{fig:various values}
\end{figure}
However, we have not been able to prove this. Hence, we pose the following conjecture:
\begin{conj}\label{conj: j mono}
    \begin{enumerate}[label = (\roman*)]
        \item For all $\theta \in \prb{\frac{\pi}{2},1.9}$, the function $x\mapsto\abs{j\pr{x+0.75i}-j\pr{e^{i\theta}}}$ is decreasing on $\prb{0,0.5}$.
        \item For all $\theta\in \prb{1.9,\frac{2\pi}{3}}$, the function $x\mapsto\abs{j\pr{x+0.65i}-j\pr{e^{i\theta}}}$ is decreasing on $\prb{0,0.5}$.
    \end{enumerate}
\end{conj}

Assuming Conjecture \ref{conj: j mono} would simplify the proof of Proposition \ref{lemma:j bounds}. Since we do not know the conjecture at this time, instead of using monotonicity, we approximate $j$ with short trigonometric polynomials and obtain the bounds with the assistance of a computer.

\begin{lemma}\label{lemma: approximate j}
    For all $a>0$, and any integer $M\ge 1$, we denote 
    \[f_{M,a}\pr{x}=\sum_{n=-1}^{M}c\pr{n}e^{-2\pi an}e^{2\pi inx},\]
    where $c(n)$ are the coefficients of the $j$-function.
    For any $a>0$ and any integer $M>\frac{1}{a^2}$,
    \begin{equation*}
        \abs{j\pr{x+ai}-f_{M,a}\pr{x}}<
        \frac{e^{2\pi\pr{\frac{1}{a} -a\pr{\sqrt{M}-\frac{1}{a}}^{2}}}}{2\sqrt{2}\pi\pr{M+1}^{3/4}}\frac{\sqrt{M}}{a\sqrt{M}-1}.
    \end{equation*}
\end{lemma}

\begin{proof}
    Let $a>0$, and let $M>\frac{1}{a^2}$ an integer. For all $n\ge 1$ we have $1\le c\pr{n}\le \frac{e^{4\pi\sqrt{n}}}{\sqrt{2}n^{3/4}}$. Hence,
    \begin{multline*}
        \abs{j\pr{x+ai}-f_{M,a}\pr{x}} \le \sum_{n=M+1}^{\infty}c\pr{n}e^{-2\pi an} \le \sum_{n=M+1}^{\infty}\frac{e^{4\pi\sqrt{n}}}{\sqrt{2}n^{3/4}}e^{-2\pi an}\\ \le \frac{1}{\sqrt{2}\pr{M+1}^{3/4}}\sum_{n=M+1}^{\infty}e^{4\pi\sqrt{n}-2\pi an}.
    \end{multline*}
    For all $n\ge 0$, we have \[4\pi\sqrt{n}-2\pi an = -2\pi a\pr{\sqrt{n}-\frac{1}{a}}^{2}+\frac{2\pi}{a}\]
    and therefore 
    \begin{equation}\label{eq: sum bound on japprox}
        \abs{j\pr{x+ai}-f_{M,a}\pr{x}}\le \frac{e^{\frac{2\pi}{a}}}{\sqrt{2}\pr{M+1}^{3/4}}\sum_{n=M+1}^{\infty}e^{-2\pi a\pr{\sqrt{n}-\frac{1}{a}}^2}.
    \end{equation}
    The function $x\mapsto e^{-2\pi a\pr{\sqrt{x}-\frac{1}{a}}^2}$ is decreasing on the interval $\pr{\frac{1}{a^2},\infty}$, and since $M>\frac{1}{a^2}$ we can bound the sum on the RHS of \eqref{eq: sum bound on japprox} by the integral from $M$ to $\infty$:
    \begin{equation}\label{eq: int-sum bound on japprox}
        \sum_{n=M+1}^{\infty}e^{-2\pi a\pr{\sqrt{n}-\frac{1}{a}}^2}<\int_{M}^{\infty} e^{-2\pi a\pr{\sqrt{x}-\frac{1}{a}}^2}dx.
    \end{equation}
    Now, 
    \begin{multline*}
        \int_{M}^{\infty} e^{-2\pi a\pr{\sqrt{x}-\frac{1}{a}}^2}dx  = \frac{-1}{2\pi a}\int_{M}^{\infty}\frac{-2\pi a}{\sqrt{x}}\pr{\sqrt{x}-\frac{1}{a}}e^{-2\pi a\pr{\sqrt{x}-\frac{1}{a}}^2}dx \\+ \frac{1}{a}\int_{M}^{\infty}\frac{1}{\sqrt{x}}e^{-2\pi a\pr{\sqrt{x}-\frac{1}{a}}^2}dx.
    \end{multline*}
    Evaluating the first integral on the right-hand side, we get:
    \[ \frac{-1}{2\pi a}\int_{M}^{\infty}\frac{-2\pi a}{\sqrt{x}}\pr{\sqrt{x}-\frac{1}{a}}e^{-2\pi a\pr{\sqrt{x}-\frac{1}{a}}^2}dx = \frac{e^{-2\pi a\pr{\sqrt{M}-\frac{1}{a}}^2}}{2\pi a},\]
    and one can easily obtain: 
    \begin{multline*}
        \frac{1}{a}\int_{M}^{\infty}\frac{1}{\sqrt{x}}e^{-2\pi a\pr{\sqrt{x}-\frac{1}{a}}^2}dx =\frac{2}{a}\int_{\sqrt{M}-\frac{1}{a}}^{\infty}e^{-2\pi at^2}dt \\\le \frac{1}{2\pi a^2\pr{\sqrt{M}-\frac{1}{a}}}\int_{\sqrt{M}-\frac{1}{a}}^{\infty}4\pi at e^{-2\pi at^2}dt = \frac{e^{-2\pi a\pr{\sqrt{M}-\frac{1}{a}}^2}}{2\pi a\pr{a\sqrt{M}-1}}.
    \end{multline*}
    Hence,
    \begin{multline}\label{eq: int bound on japprox}
        \int_{M}^{\infty} e^{-2\pi a\pr{\sqrt{x}-\frac{1}{a}}^2}dx \le \frac{e^{-2\pi a\pr{\sqrt{M}-\frac{1}{a}}^2}}{2\pi a}+\frac{e^{-2\pi a\pr{\sqrt{M}-\frac{1}{a}}^2}}{2\pi a\pr{a\sqrt{M}-1}}\\ = \frac{e^{-2\pi a\pr{\sqrt{M}-\frac{1}{a}}^2}}{2\pi a}\frac{a\sqrt{M}}{a\sqrt{M}-1}.
    \end{multline}
     Finally, we combine \eqref{eq: sum bound on japprox}, \eqref{eq: int-sum bound on japprox}, and \eqref{eq: int bound on japprox} to obtain
     \[\abs{j\pr{x+ai}-f_{M,a}\pr{x}}<
        \frac{e^{2\pi\pr{\frac{1}{a} -a\pr{\sqrt{M}-\frac{1}{a}}^{2}}}}{2\sqrt{2}\pi\pr{M+1}^{3/4}}\frac{\sqrt{M}}{a\sqrt{M}-1}.\qedhere\]
\end{proof}

\begin{proof}[Proof of Proposition \ref{lemma:j bounds}]
Since the $j$-function is real on the arc, the function $x\mapsto \abs{j(x+iy)-j(e^{i\theta})}$ is an even function. Thus,
\begin{equation}\label{j diff}
    \min_{\abs{x}\le \frac{1}{2}}\abs{j(x+iy)-j(e^{i\theta})} = \min_{x\in\left[0,0.5\right]}\abs{j(x+iy)-j(e^{i\theta})}.
\end{equation}
Furthermore, since $j$ is injective on the fundamental domain and real on its boundary, the function $\theta\mapsto j(e^{i\theta})$ is decreasing on $\left[\frac{\pi}{2},\frac{2\pi}{3}\right]$, as $j(i)=1728$ and $j(\rho)=0$. Hence, our goal is to find a lower bound to the expressions
\[\min_{t\in\prb{j\pr{e^{1.9i}},1728}}\min_{x\in \prb{-0.5,0.5}}\abs{j\pr{x+0.75i}-t},\]
and 
\[\min_{t\in\prb{0,j\pr{e^{1.9i}}}}\min_{x\in \prb{-0.5,0.5}}\abs{j\pr{x+0.65i}-t}.\] 
Let us estimate $j\pr{e^{1.9}}$: Using Lemma \ref{lemma: approximate j} with $a=\sin\pr{1.9}$ and $M=6$, we get
\[\abs{j\pr{e^{1.9i}}-\Re\pr{f_{6,\sin\pr{1.9}}\pr{\cos\pr{1.9}}}}\le \abs{j\pr{e^{1.9i}}-{f_{6,\sin\pr{1.9}}\pr{\cos\pr{1.9}}}}<4\cdot 10^{-4}.\]
We have
\begin{multline*}
    \Re\pr{f_{6,\sin\pr{1.9}}\pr{\cos\pr{1.9}}} = \sum_{n=-1}^{6}c\pr{n}e^{-2\pi n\sin\pr{1.9}}\cos\pr{2\pi n\cos\pr{1.9}} \\
    = 271.09885 \ldots.
\end{multline*}
Hence,
\begin{equation}\label{eq: j bounds 1.9}
    271 \le j\pr{e^{1.9i}}\le 272.
\end{equation}

\begin{enumerate}[label = (\roman*)]
\item \emph{The case $y=0.75$ and $\theta \in \left[\frac{\pi}{2},1.9\right]$:}
In this case $j(e^{i\theta})\in\left[j(e^{1.9i}),1728\right]$ and by \eqref{eq: j bounds 1.9} we know $j(e^{1.9i})\ge 271$. Using Lemma \ref{lemma: approximate j} with $a=0.75$ and $M=5$, we get 
\[\abs{j\pr{x+0.75i} - f_{5,0.75}\pr{x}} < 10.\]
Denote $f\pr{x}=f_{5,0.75}\pr{x}$, we have 
\begin{multline*}
    f\pr{x} = e^{\frac{3\pi}{2}}e^{-2\pi ix}+744+196884e^{-\frac{3\pi}{2}}e^{2\pi ix}\\
    +21493760e^{-{3\pi}}e^{4\pi ix}+864299970e^{-\frac{9\pi}{2}}e^{6\pi ix}\\+20245856256e^{-{6\pi}}e^{8\pi ix}+333202640600e^{-\frac{15\pi}{2}}e^{10\pi ix}.
\end{multline*}
Examining the plot of the real and imaginary parts of $f(x)$ on $\left[0,0.5\right]$ (see Figure \ref{fig:jontheline075}), we can see that a bound to \eqref{j diff} is achievable by subdividing the interval $\left[0,0.5\right]$.
\begin{figure}[ht]
    \begin{center}
    \includegraphics[height=65mm]{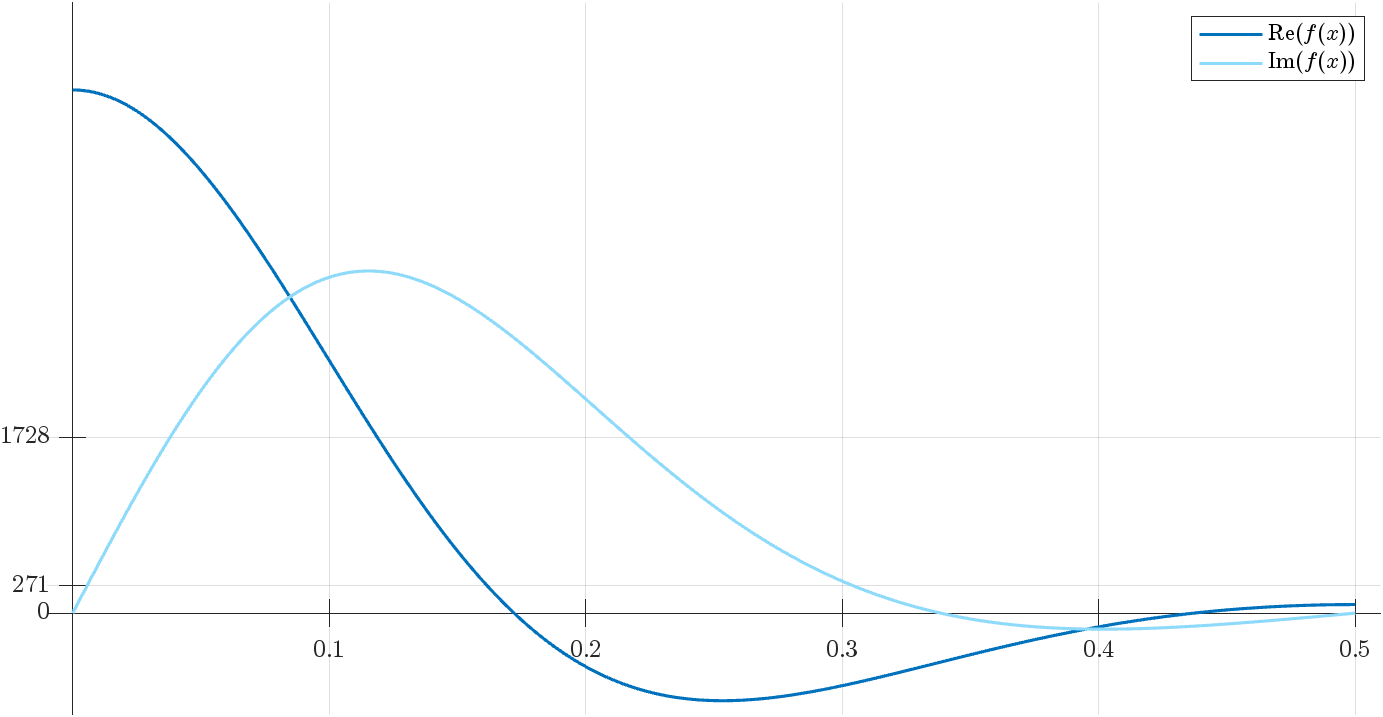}
    \caption{The real and imaginary parts of $f(x)$ on $\left[0,0.5\right]$}
    \label{fig:jontheline075}
    \end{center}
\end{figure}

To begin, we will prove some monotonicity properties on the real part of $f$.
\begin{lemma}\label{real part monotone}
 There exists $x_0\in\pr{0,0.5}$ such that $\Re f$ is monotonically decreasing on $\prb{0,x_0}$ and monotonically increasing on $\prb{x_0,0.5}$. Furthermore,  $x_0$ is approximately $0.253311$.
\end{lemma}

\begin{proof} For all $x\in\prb{0,0.5}$, we have $\Re f(x)=\sum_{n=0}^{5}a\pr{n}\cos(n\cdot 2\pi x)$ where 
        \begin{align*}
            a\pr{0}& =744, & a\pr{1} & = e^{\frac{3\pi}{2}}+196884e^{-\frac{3\pi}{2}},\\
            & \forall n\ge 2 & a\pr{n} & = e^{-\frac{3\pi}{2}n}c\pr{n}.
        \end{align*}
        For any $x\in\R$, we have $T_{n}\pr{\cos\pr{2\pi x}}=\cos(n\cdot 2\pi x)$ where $T_n$ is the Chebyshev polynomial of the first kind. Hence, 
        \[\Re f(x)=\sum_{n=0}^{5}a\pr{n}\cos(n\cdot 2\pi x)=\sum_{n=0}^{5}a\pr{n}T_{n}\pr{z}, \quad z=\cos\pr{2\pi x}.\]
        It is a polynomial of degree $5$ and
        \[\Re f(x)=311.893 z^5+1054.79 z^4+2116.42 z^3+2414.28 z^2+97.735 z-858.687.\]
        Our goal is to understand the monotonicity of the polynomial on the right-hand side in the interval $\prb{-1,1}$. We will use Goursat's Lemma (see \cite{PowersReznick}): The Goursat transform of a polynomial $P$ is the polynomial 
        \[\Tilde{P}\pr{z} = \pr{1+z}^{\deg P}P\pr{\frac{1-z}{1+z}}.\]
        The map $z\mapsto\frac{1-z}{1+z}$ decreases from $1$ to $-1$ as $z$ increases from $0$ to $\infty$, and $\pr{1+z}^{\deg P}>0$ for all $z\ge 0$. Goursat's Lemma states that the signs of $\Tilde{P}\pr{z_0}$ and $P\pr{\frac{1-z_0}{1+z_{0}}}$ are equal and  $\Tilde{P}\pr{z_0}=0$ if and only if $P\pr{\frac{1-z_0}{1+z_{0}}}=0$.
        Denote
        \[p\pr{z}=\frac{d}{dz}\Re f(x) = 1559.47 z^4+4219.18 z^3+6349.26 z^2+4828.55 z+97.735,\]
        and consider the Goursat transform of $p$, i.e.
        \begin{multline*}
            \Tilde{p}\pr{z} = \pr{1+z}^{4}p\pr{\frac{1-z}{1+z}} \\= -1041.27 z^4-7065.67 z^3-2755.32 z^2-4628.17 z+17054.2
        \end{multline*}
        The coefficients of $\Tilde{p}$ are all negative except for the constant term, which is positive. Hence, $\Tilde{p}$ is decreasing on $\left[0,\infty\right)$. We have $\Tilde{p}\pr{0}=17054.2>0$ and $\lim_{z\to\infty}p\pr{z}=-\infty$, therefore, there exists a unique zero for $\Tilde{p}$ in $\left[0,\infty\right)$ at $z_0$. Computing the roots, we get that $z_0$ is approximately $1.0424883$.
        Recall that the signs of $\Tilde{p}\pr{z}$ and $p\pr{\frac{1-z}{1+z}}$ are equal and  $\Tilde{p}\pr{z}=0$ if and only if $p\pr{\frac{1-z}{1+z}}=0$. Finally, we have $\frac{1-z_0}{1+z_{0}} = -0.0208023$ and thus $p$ is positive on $\prb{-1,-0.0208023}$ and negative on $\prb{-0.0208023,1}$. Therefore, $p$ is increasing on $\prb{-1,-0.0208023}$ and decreasing on $\prb{-0.0208023,1}$.
        Denote $x_0=\frac{1}{2\pi}\arccos\pr{\frac{1-z_0}{1+z_{0}}}= 0.253311\ldots$. The function $z\mapsto\frac{1}{2\pi}\arccos\pr{z}$ is decreasing from $0.5$ to $0$ on $\prb{-1,1}$, and $\Re f\pr{\frac{1}{2\pi}\arccos\pr{z}} = p\pr{z}$ for any $z\in\prb{-1,1}$, thus, $\Re f$ is decreasing on $\prb{0,x_0}$ and increasing on $\prb{x_0,0.5}$.
\end{proof}
Let us subdivide the interval into three parts: $\left[0,0.1\right]$, $\left[0.1,0.2\right]$, and $\prb{0.2,0.5}$. 
\begin{itemize}
    \item $x\in\left[0,0.1\right]$:
    By Lemma \ref{real part monotone}, the function $\Re f$ is decreasing on $\prb{0,0.1}$ and therefore $\Re f\pr{x} \ge \Re f\pr{0.1} = 2481.16>2000$. Thus,
    \begin{equation*}
        \Re(j(x+0.75i)) \ge \Re(f(x)) - 10 > 2000 -10 =1990.
    \end{equation*}
    Hence, for all $x\in\left[0,0.1\right]$:
    \begin{multline*}
        \abs{j(x+0.75i)-j(e^{i\theta})} \ge \abs{\Re(j(x+0.75i)-j(e^{i\theta}))} \\
        \ge \Re(j(x+0.75i))-j(e^{i\theta}) \ge 1990 - 1728 = 262.
    \end{multline*}
    \item $x\in\left[0.1,0.2\right]$:
        In this case, one needs to consider the imaginary part of $f$:
        \[\Im f\pr{x} = \pr{196884e^{-\frac{3\pi}{2}}-e^{\frac{3\pi}{2}}}\sin\pr{2\pi x}+\sum_{n=2}^{5}e^{-\frac{3\pi}{2}n}c\pr{n}\sin\pr{2\pi nx},\]
        By examining $\sin\pr{2\pi n x}$ on the the intervals $\prb{0.1,0.2}$ for $n\in\prs{1,\ldots,5}$ we can deduce 
        \begin{align*}
            \sin\pr{2\pi x} & \ge \sin\pr{\frac{\pi}{5}}\\
            \sin\pr{2\pi x} & \ge \sin\pr{\frac{4\pi}{5}} \\
            \sin\pr{2\pi x} & \ge \sin\pr{\frac{6\pi}{5}}\\
            \sin\pr{2\pi x} & \ge -1\\
            \sin\pr{2\pi x} & \ge -1
        \end{align*}
        Therefore, 
        \begin{multline*}
            \Im f\pr{x} = \pr{196884e^{-\frac{3\pi}{2}}-e^{\frac{3\pi}{2}}}\sin\pr{2\pi x}+\sum_{n=2}^{5}e^{-\frac{3\pi}{2}n}c\pr{n}\sin\pr{2\pi nx} \\\ge \pr{196884e^{-\frac{3\pi}{2}}-e^{\frac{3\pi}{2}}}\sin\pr{\frac{\pi}{5}} + 21493760 e^{-3 \pi }\sin\pr{\frac{4\pi}{5}} \\+ 864299970e^{-\frac{9\pi}{2}}\sin\pr{\frac{6\pi}{5}} - 20245856256 e^{-6 \pi } \\- 333202640600 e^{-\frac{15 \pi }{2}}=1474.07.
        \end{multline*}
        Hence, $\Im(f(x)) > 1400$, which yields
    \begin{equation*}
        \Im(j(x+0.75i)) \ge \Im(f(x)) -10 > 1400 -10 =1390.
    \end{equation*}
    Hence, for all $x\in\left[0.1,0.2\right]$:
    \begin{equation*}
        \abs{j(x+0.75i)-j(e^{i\theta})} \ge \Im(j(x+0.75i)) \ge 1390.
    \end{equation*}
    \item  $x\in\left[0.2,0.5\right]$:
    By Lemma \ref{real part monotone}, the function $\Re f$ is decreasing on $\prb{0.2,x_0}$ and increasing on $\prb{x_0,0.5}$. Therefore, 
    \[\Re f\pr{x}\le \max{\Re f\pr{0.2},\Re f\pr{0.5}}.\]
    We have $\Re f\pr{0.2} = $ and $\Re f\pr{0.5}=84.3362$. Hence, $\Re f\pr{x} \le 85$. Thus,
    \begin{equation*}
        \Re(j(x+0.75i)) \le \Re(f(x)) + 10 < 85+10 =95.
    \end{equation*}
    Hence, for all $x\in\left[0.2,0.5\right]$:
    \begin{multline*}
        \abs{j(x+0.75i)-j(e^{i\theta})} \ge \abs{\Re(j(x+0.75i)-j(e^{i\theta}))} \\
        \ge j(e^{i\theta})-\Re(j(x+0.75i)) \ge 271 - 95 = 176.
    \end{multline*}
\end{itemize}
Finally, we can conclude that for all $\theta \in \left[\frac{\pi}{2},1.9\right]$: 
\begin{equation}
    \min_{x\in \left[-0.5,0.5\right]}\abs{j(x+0.75i)-j(e^{i\theta})}\ge 176.\tag{\ref{eq: j-diff 75}}
\end{equation}

\item \emph{The case $y=0.65$ and $\theta \in \left[1.9,\frac{2\pi}{3}\right]$:}
In this case $j(e^{i\theta})\in\left[0,j(e^{1.9i})\right]$ and by \eqref{eq: j bounds 1.9} we have $j(e^{1.9i})\le 272$. Using Lemma \ref{lemma: approximate j} with $a=0.65$ and $M=7$, we get 
\[\abs{j\pr{x+0.65i} - f_{7,0.65}\pr{x}} < 10.\]
showing \[\abs{j(x+0.65i)} > \abs{f_{7,0.65}\pr{x}} -10.\]
Denote $g\pr{x} = f_{7,0.65}\pr{x}$, then we have
\begin{multline*}
    g\pr{x} = e^{\frac{13\pi}{10}}e^{-2\pi ix} + 744 + 196884e^{-\frac{13\pi}{10}}e^{2\pi ix} + 21493760e^{-\frac{13\pi}{5}}e^{4\pi ix}\\
    + 864299970e^{-\frac{39\pi}{10}}e^{6\pi ix} +20245856256e^{-\frac{26\pi}{5}}e^{8\pi ix}
    + 333202640600e^{-\frac{13\pi}{2}}e^{10\pi ix}\\ +4252023300096e^{-\frac{39\pi}{5}}e^{12\pi ix} + 44656994071935e^{-\frac{91\pi}{10}}e^{14\pi ix}.
\end{multline*}
Write $g\pr{x} = \sum_{n=-1}^{7}a\pr{n}e^{2\pi inx}$ where $a\pr{n} = e^{-\frac{13\pi}{10}n}c\pr{n}$, then 
\begin{align*}
    \abs{g\pr{x}}^2  & = \pr{\sum_{n=-1}^{7}a\pr{n}e^{2\pi i nx}}\overline{\pr{\sum_{m=-1}^{7}a\pr{m}e^{2\pi i mx}}} \\
    & = \sum_{n,m=-1}^{7}a\pr{n}a\pr{m}e^{2\pi i\pr{n-m}x}\\
    & =\sum_{k=-8}^{8}\sum_{\begin{subarray}{c}
        m=-1\\
        -1\le m+k\le 7
    \end{subarray}}^{8} a\pr{k+m}a\pr{m}e^{2\pi k x}\\
    & =\sum_{m=-1§}^{7}a\pr{m}^{2} + \sum_{k=1}^{8}\sum_{m=-1}^{7-k}a\pr{m}a\pr{m+k}e^{2\pi i kx}+\sum_{k=1}^{8}\sum_{m=k-1}^{7}a\pr{m}a\pr{m-k}e^{-2\pi i kx} \\
    & = \sum_{m=-1§}^{7}a\pr{m}^{2} + \sum_{k=1}^{8}\sum_{m=-1}^{7-k}a\pr{m}a\pr{m+k}\pr{e^{2\pi ikx}+e^{-2\pi ikx}}\\
    & =\sum_{m=0}^{7}a\pr{n}^{2}+2\sum_{k=1}^{8}\sum_{m=-1}^{7-k}a\pr{m}a\pr{m+k}\cos\pr{2\pi kx}.
\end{align*}
Denote $b\pr{0} = \sum_{m=0}^{7}a\pr{n}^{2}$ and $b\pr{k}= 2\sum_{m=-1}^{7-k}a\pr{m}a\pr{m+k}$. Recall that for any $x\in\R$ we have $T_{n}\pr{\cos\pr{2\pi x}}=\cos(n\cdot 2\pi x)$ where $T_n$ is the Chebyshev polynomial of the first kind and write 
\[\abs{g\pr{x}}^2 = \sum_{k=0}^{8}b\pr{k}T_{k}\pr{z}, \quad z = \cos\pr{2\pi x}.\]
Denote $p\pr{z} = \sum_{k=0}^{8}b\pr{k}T_{k}\pr{z}$ then 
\begin{multline*}
    p\pr{z} = 260611.69 z^8+2369295.32 z^7+9445879.6 z^6\\+23312894.25 z^5+43183990.77 z^4+63793471.45 z^3\\+69169139.18 z^2+4.7012209.65z+14780542.62.
\end{multline*}
As before, consider the Goursat transform of $p'$, 
\begin{multline*}
    \widetilde{p'}\pr{z} = \pr{1+z}^{7}p'\pr{\frac{1-z}{1+z}}  = 1707749.94 z^7+23929635.67 z^6\\+1.50914350.75z^5+571404726.126z^4+1339181961.02z^3\\+1805503665.81z^2+1383544170.65z+741376575.025
\end{multline*}
All of the coefficients of $\widetilde{p'}$ are positive, hence $\widetilde{p'}$ is positive on $\left[0,\infty\right)$. Thus, by Goursat's Lemma $p'$ is positive on $\prb{-1,1}$ and therefore $p$ is an increasing function on $\prb{-1,1}$. Since $\abs{g\pr{x}}^2 = p\pr{\frac{1}{2\pi}\arccos\pr{x}}$ and $\frac{1}{2\pi}\arccos\pr{x}$ decreases from $1$ to $-1$ on $\prb{0,0.5}$, the function $x\mapsto\abs{g\pr{x}}^{2}$ is decreasing on $\prb{0,0.5}$. Hence, the function $x\mapsto\abs{g\pr{x}}$ is decreasing on $\prb{0,0.5}$.
Therefore, 
\[\abs{g\pr{x}} \ge \abs{g\pr{0.5}}= 593.543\]
Hence,
\begin{equation}
    \abs{j(x+0.65i)-j(e^{i\theta})} \ge  \abs{j(x+0.65i)} - j(e^{i\theta}) > 593 -10 - 272 = 311. \tag{\ref{eq: j-diff 65}}\qedhere
\end{equation}
\begin{figure}[ht]
    \begin{center}
    \includegraphics[height=75mm]{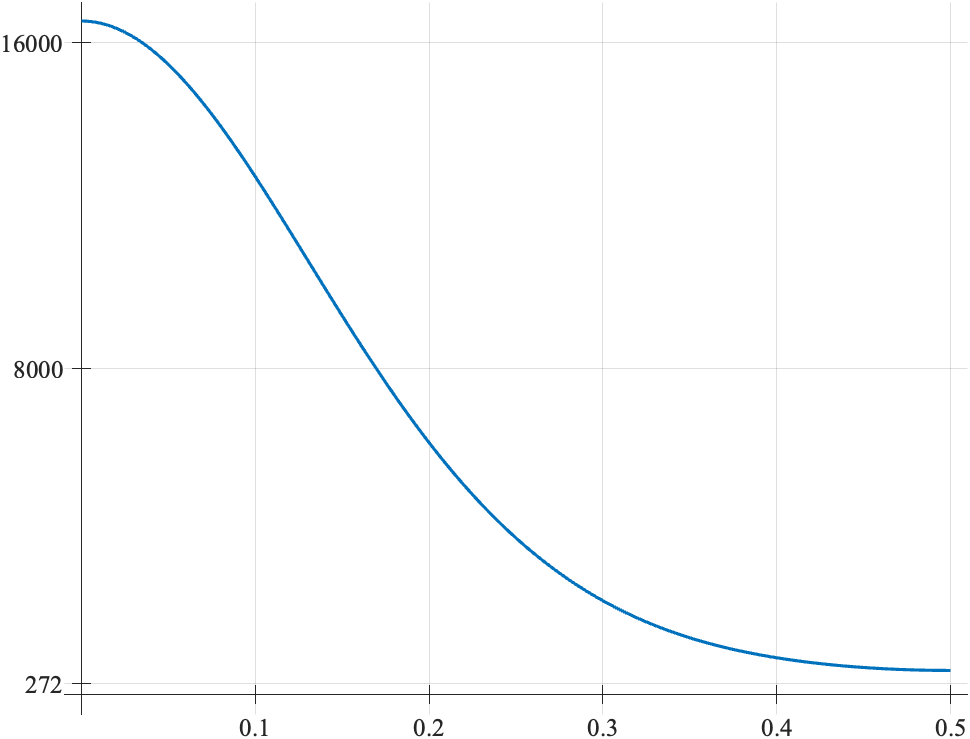}
    \caption{$\abs{g(x)}$ on $\left[0,0.5\right]$}
     \label{fig:6}
    \end{center}
\end{figure}
\end{enumerate}
\end{proof}

\subsection{Monotonicity of $\abs{E_{4}\pr{e^{i\theta}}}$ and $\abs{E_{6}\pr{e^{i\theta}}}$}\label{E_k monotinicity} Recall that for any even $k\ge 4$, we define \[e_{k}\pr{\theta} = e^{\frac{ik\theta}{2}}E_{k}\pr{e^{i\theta}},\]
and for $k=2$, we define \[e_{2}\pr{\theta}=e^{i\theta}E_{2}\pr{e^{i\theta}}+\frac{3}{i\pi}.\]
\begin{lemma}
    \begin{enumerate}[label = (\roman*)]
        \item The function $e_{6}$ is positive for all $\theta\in\left(\frac{\pi}{2},\frac{2\pi}{3}\right]$ and vanishes at $\frac{\pi}{2}$.
        \item The function $\theta\mapsto\abs{E_{4}\pr{e^{i\theta}}}$ is a decreasing function on $\left[\frac{\pi}{2},\frac{2\pi}{3}\right]$. In particular, for all $\theta\in\left[\frac{\pi}{2},1.9\right]$ we have 
        \begin{equation}\label{eq: An upper bound for E_4 arc 0.75}
            \abs{E_{4}\pr{e^{i\theta}}}\le E_{4}\pr{i} = \frac{3\varpi^{4}}{\pi^{4}} = 1.455761\ldots,
        \end{equation}
        with $\varpi = 2\int_{0}^{1}\frac{dx}{\sqrt{1-x^{4}}} = 2.622057\ldots $, and for all $\theta\in\left[1.9, \frac{2\pi}{3}\right]$
        \begin{equation}\label{eq: An upper bound for E_4 arc 0.65}
            \abs{E_{4}\pr{e^{i\theta}}}\le \abs{E_{4}\pr{e^{1.9i}}}=0.900253\ldots.
        \end{equation}
        \item The function $\theta\mapsto\abs{E_{6}\pr{e^{i\theta}}}$ is an increasing function on $\left[\frac{\pi}{2},\frac{2\pi}{3}\right]$. In particular, for all $\theta\in\left[\frac{\pi}{2},1.9\right]$ we have 
        \begin{equation}\label{eq: An upper bound for E_6 arc 0.75}
            \abs{E_{6}\pr{e^{i\theta}}}\le \abs{E_{6}\pr{e^{1.9i}}} =1.980151\ldots,
        \end{equation}
        and for all $\theta\in\left[1.9, \frac{2\pi}{3}\right]$
        \begin{equation}\label{eq: An upper bound for E_6 arc 0.65}
            \abs{E_{6}\pr{e^{i\theta}}}\le \abs{E_{6}\pr{\rho}}=\frac{27\varpi'^{6}}{2\pi^{6}}=2.881536\ldots.
        \end{equation}
        with $\varpi' = 2\int_{0}^{1}\frac{dx}{\sqrt{1-x^{6}}} =2.622057\ldots $.
    \end{enumerate}
\end{lemma}

\begin{proof}
    \begin{enumerate}[label = (\roman*)]
        \item $E_{6}$ has a unique zero in the fundamental domain, at $i = e^{\pi i/2}$. Therefore, $e_{6}$ is continuous and nonzero for all $\theta\in\left(\frac{\pi}{2},\frac{2\pi}{3}\right]$. Hence, it is enough to show that $e_{6}\pr{\frac{2\pi}{3}}>0$. As we previously saw, by \cite{Katayama}
        \[e_{6}\pr{\frac{2\pi}{3}}=e^{2\pi i}E_{6}\pr{\rho}=\frac{27\varpi'^{6}}{2\pi^{6}}>0.\]
        \item Since $e_4$ is negative, we have $\abs{E_{4}\pr{e^{i\theta}}} = -e_{4}\pr{\theta}$. Hence, it is enough to show $\frac{de_{4}}{d\theta}>0$. Indeed, recall the following Ramanujan's identity,
        \[\frac{1}{2\pi i}\frac{dE_{4}}{d\tau} =\frac{E_{2}E_{4}-E_{6}}{3}.\]
        Then, 
        \begin{align*}
            \frac{de_{4}}{d\theta}\pr{\theta} & = \frac{d}{d\theta}\pr{e^{2i\theta}E_{4}\pr{e^{i\theta}}}\\
            & = 2ie^{2i\theta}E_{4}\pr{e^{i\theta}}+e^{2i\theta}\frac{d}{d\theta}\pr{E_{4}\pr{e^{i\theta}}}\\
            & = 2ie_{4}\pr{\theta}-2\pi e^{2i\theta}e^{i\theta}\frac{dE_{4}}{d\tau}\pr{e^{i\theta}}\\
            & = 2ie_{4}\pr{\theta}-\frac{2\pi}{3}e^{2i\theta}e^{i\theta}\pr{E_{2}\pr{e^{i\theta}}E_{4}\pr{e^{i\theta}}-E_{6}\pr{e^{i\theta}}}\\
            & = 2ie_{4}\pr{\theta}-\frac{2\pi}{3}e_{4}\pr{\theta}\pr{e_{2}\pr{\theta}-\frac{3}{i\pi}} + \frac{2\pi}{3}e_{6}\pr{e^{i\theta}}\\
            & = \frac{2\pi}{3}\pr{e_{6}\pr{\theta}-e_{4}\pr{\theta}e_{2}\pr{\theta}}
        \end{align*}
        By (i), $e_6>0$ and as we saw in Lemma \ref{lemma: e_2,e_4, delta are negative}, $e_{2}<0$ and $e_{4}<0$ which shows $\frac{de_{4}}{d\theta}>0$, and thus $e_{4}$ is an increasing function. The bound \eqref{eq: An upper bound for E_4 arc 0.75} follows from a result which is due to Hurwitz \cite{Hurwitz},
        \[E_{4}\pr{i}=\frac{3\varpi^{4}}{\pi^{4}}.\]
        The bound \eqref{eq: An upper bound for E_4 arc 0.65} follows from a simple calculation: 
        \[e_{4}\pr{1.9}=-0.900253\ldots.\]
        \item Since $e_6$ is positive, we have $\abs{E_{6}\pr{e^{i\theta}}} = e_{6}\pr{\theta}$. Hence, it is enough to show $\frac{de_{6}}{d\theta}>0$. Indeed, recall the following Ramanujan's identity,
        \[\frac{1}{2\pi i}\frac{dE_{6}}{d\tau} =\frac{E_{2}E_{6}-E_{4}^{2}}{2}.\]
        Then, 
        \begin{align*}
            \frac{de_{6}}{d\theta}\pr{\theta} & = \frac{d}{d\theta}\pr{e^{3i\theta}E_{6}\pr{e^{i\theta}}}\\
            & = 3ie^{3i\theta}E_{4}\pr{e^{i\theta}}+e^{3i\theta}\frac{d}{d\theta}\pr{E_{6}\pr{e^{i\theta}}}\\
            & = 3ie_{6}\pr{\theta}-2\pi e^{3i\theta}e^{i\theta}\frac{dE_{6}}{d\tau}\pr{e^{i\theta}}\\
            & = 3ie_{6}\pr{\theta}-\pi e^{3i\theta}e^{i\theta}\pr{E_{2}\pr{e^{i\theta}}E_{6}\pr{e^{i\theta}}-E_{4}\pr{e^{i\theta}}^{2}}\\
            & = 3ie_{6}\pr{\theta}-\pi e_{6}\pr{\theta}\pr{e_{2}\pr{\theta}-\frac{3}{i\pi}} +\pi e_{4}\pr{e^{i\theta}}^{2}\\
            & = \pi\pr{e_{4}\pr{\theta}^{2}-e_{6}\pr{\theta}e_{2}\pr{\theta}}
        \end{align*}
        By (i), $e_6>0$ and as we saw in Lemma \ref{lemma: e_2,e_4, delta are negative}, $e_{2}<0$. Hence, $\frac{de_{6}}{d\theta}>0$, which shows that $e_{6}$ is an increasing function. The bound \eqref{eq: An upper bound for E_6 arc 0.75} follows from a simple calculation:
        \[e_6(1.9) = 1.980151\ldots.\]
        The bound \eqref{eq: An upper bound for E_6 arc 0.65}  
        follows from a generalization of Hurwitz's result, which is due to Katayama \cite{Katayama},
        \[E_{6}\pr{\rho}=\frac{27\varpi^{6}}{2\pi^{6}}.\qedhere\]
    \end{enumerate} 
\end{proof}
\subsection{Bounds for $E_4$}\label{E_4 bounds}
\begin{lemma}
    \begin{enumerate}[label = (\roman*)]
        \item For all $x\in\left[-0.5,0.5\right]$, 
        \begin{equation}\label{eq: An upper bound for |E_4(x+0.65i)|}
            \left|E_{4}\left(x+0.65i\right)\right| < 5.9.
        \end{equation}
        \item For all $x\in\left[-0.5,0.5\right]$, 
        \begin{equation}\label{eq: An upper bound for |E_4(x+0.75i)|}
            \left|E_{4}\left(x+0.75i\right)\right| < 3.45.
        \end{equation}
    \end{enumerate}
\end{lemma}
\begin{proof}
\begin{enumerate}[label = (\roman*)]
\item \emph{An upper bound for $|E_{4}(x+0.65i)|$:}
For all $n\ge 3$ we have $n^{4}e^{-\frac{13\pi}{20}n}\le \frac{3}{10}$, and thus
\begin{multline*}
      \left|\left|E_{4}\left(x+0.65i\right)\right|-\left|1+240e^{-\frac{13\pi}{10}}e^{2\pi ix}+2160e^{-\frac{13\pi}{5}}e^{4\pi ix}\right|\right| \le 240\sum_{n=3}^{\infty}\sigma_{3}(n)e^{-\frac{13\pi}{10}n} \\ \le 240\sum_{n=3}^{\infty}n^4e^{-\frac{13\pi}{10}n} \le 240\cdot\frac{3}{10}\sum_{n=3}^{\infty}e^{-\frac{13\pi}{20}n} = 72\frac{e^{-\frac{39\pi}{20}}}{1-e^{-\frac{13\pi}{20}}}<\frac{1}{5}=0.2.
\end{multline*}
Now, we have 
\begin{equation*}
    \left|1+240e^{-\frac{13\pi}{10}}e^{2\pi ix}+2160e^{-\frac{13\pi}{5}}e^{4\pi ix}\right| \le 1+240e^{-\frac{13\pi}{10}}+2160e^{-\frac{13\pi}{5}} < 5.7
\end{equation*}
Hence,
\begin{equation}
    \left|E_{4}\left(x+0.65i\right)\right| \le 5.9.\tag{\ref{eq: An upper bound for |E_4(x+0.65i)|}}
\end{equation}

\item \emph{An upper bound for $|E_{4}(x+0.75i)|$:}
For all $n\ge 3$ we have $n^{4}e^{-\frac{3\pi}{4}}\le \frac{1}{5}$, and thus
\begin{multline*}
      \left|\left|E_{4}\left(x+0.75i\right)\right|-\left|1+240e^{-\frac{2\pi}{3}}e^{2\pi ix}+2160e^{-3\pi ix}\right|\right| \le 240\sum_{n=3}^{\infty}\sigma_{3}(n)e^{-\frac{2\pi}{3}n} \\ \le 240\sum_{n=7}^{\infty}n^4e^{-\frac{2\pi}{3}n} \le 240\cdot\frac{1}{5}\sum_{n=3}^{\infty}e^{-\frac{3\pi}{4}n} = 48\frac{e^{-\frac{9\pi}{4}}}{1-e^{-\frac{3\pi}{4}}}<\frac{1}{20}=0.05.
\end{multline*}
Now, we have 
\begin{equation*}
    \left|1+240e^{-\frac{2\pi}{3}}e^{2\pi ix}+2160e^{-3\pi}e^{4\pi ix}\right| \le 1+ 240e^{-\frac{2\pi}{3}}+2160e^{-3\pi} < 3.4
\end{equation*}
Hence,
\begin{equation}
    \left|E_{4}\left(x+0.75i\right)\right| < 3.45. \tag{\ref{eq: An upper bound for |E_4(x+0.75i)|}}\qedhere
\end{equation}
\end{enumerate}
\end{proof}

\subsection{Bounds for $E_6$}\label{E_6 bounds}
\begin{lemma}
    \begin{enumerate}[label = (\roman*)]
        \item For all $x\in\left[-0.5,0.5\right]$, 
        \begin{equation}\label{eq: An upper bound for |E_6(x+0.65i)|}
            \left|E_{6}\left(x+0.65i\right)\right| < 14.26.
        \end{equation}
        \item For all $x\in\left[-0.5,0.5\right]$, 
        \begin{equation}\label{eq: An upper bound for |E_6(x+0.75i)|}
            \left|E_{6}\left(x+0.75i\right)\right| < 5.25.
        \end{equation}
    \end{enumerate}
\end{lemma}
\begin{proof}
\begin{enumerate}[label = (\roman*)]
\item \emph{An upper bound for $|E_{6}(x+0.65i)|$:}
For all $n\ge 3$ we have $n^{6}e^{-\frac{13\pi}{20}n}\le \frac{8}{5}$, and thus
\begin{multline*}
      \left|\left|E_{6}\left(x+0.65i\right)\right|-\left|1-504e^{-\frac{13\pi}{10}}e^{2\pi ix}-16632e^{-\frac{13\pi}{5}}e^{4\pi ix}\right|\right| \le 504\sum_{n=3}^{\infty}\sigma_{6}(n)e^{-\frac{13\pi}{10}n} \\ \le 504\sum_{n=3}^{\infty}n^6 e^{-\frac{13\pi}{10}n} \le 504\cdot\frac{8}{5}\sum_{n=3}^{\infty}e^{-\frac{13\pi}{20}n} = \frac{4032}{5}\frac{e^{-\frac{39\pi}{20}}}{1-e^{-\frac{13\pi}{20}}}<2.05.
\end{multline*}
Now, we have 
\begin{multline*}
    \left|1-504e^{-\frac{13\pi}{10}}e^{2\pi ix}-16632e^{-\frac{13\pi}{5}}e^{4\pi ix}\right| = \left|\left(1-252e^{-\frac{13\pi}{10}}e^{2\pi ix}\right)^{2}-80136e^{-\frac{13\pi}{5}}e^{4\pi ix}\right|  \\ \le \abs{\left(1-252e^{-\frac{13\pi}{10}}\right)^{2}-80136e^{-\frac{13\pi}{5}}} < 12.21.
\end{multline*}
Hence,
\begin{equation}
    \left|E_{6}\left(x+0.65i\right)\right| \le 14.26. \tag{\ref{eq: An upper bound for |E_6(x+0.65i)|}}
\end{equation}

\item \emph{An upper bound for $|E_{6}(x+0.75i)|$:}
For all $n\ge 1$ we have $n^{6}e^{-\frac{3\pi}{4}n}\le \frac{7}{10}$, and thus
\begin{multline*}
      \left|\left|E_{6}\left(x+0.75i\right)\right|-\left|1-504e^{-\frac{2\pi}{3}}e^{2\pi ix}-16632e^{-3\pi}e^{4\pi ix}\right|\right| \le 504\sum_{n=3}^{\infty}\sigma_{6}(n)e^{-\frac{2\pi}{3}n} \\ \le 504\sum_{n=3}^{\infty}n^6 e^{-\frac{2\pi}{3}n} \le 504\cdot\frac{7}{10}\sum_{n=3}^{\infty}e^{-\frac{3\pi}{4}n} = \frac{4032}{5}\frac{e^{-\frac{9\pi}{4}}}{1-e^{-\frac{3\pi}{4}}}<0.35.
\end{multline*}
Now, we have 
\begin{multline*}
    \left|1-504e^{-\frac{2\pi}{3}}e^{2\pi ix}-16632e^{-3\pi}e^{4\pi ix}\right| = \left|\left(1-252e^{-\frac{2\pi}{3}}e^{2\pi ix}\right)^{2}-80136e^{-3\pi}e^{4\pi ix}\right|  \\ \le \abs{\left(1-252e^{-\frac{2\pi}{3}}\right)^{2}-80136e^{-3\pi}} < 4.9
\end{multline*}
Hence,
\begin{equation}
    \left|E_{6}\left(x+0.75i\right)\right| < 5.25. \tag{\ref{eq: An upper bound for |E_6(x+0.75i)|}}\qedhere
\end{equation}
\end{enumerate}
\end{proof}

\subsection{Proof of Theorem \ref{QuanMR}} We begin with bounding the expression
\[ \max_{\abs{x}\le\frac{1}{2}}\abs{\frac{E_{k'}\pr{e^{i\theta}}E_{14-k'}\pr{x+ai}}{\Delta\pr{x+ai}\pr{j\pr{x+ai}-j\pr{e^{i\theta}}}}},\] for $a=0.75$ and $a=0.65$ and all six different choices of $k'$. Recall the identities, 
\begin{align*}
     E_{8} & = E_{4}^2 ,\\
     E_{10} & = E_{4}E_{6}, \\
     E_{14} & = E_{4}^{2}E_{6}.
\end{align*}
For all $\theta\in\prb{\frac{\pi}{2},\frac{2\pi}{3}}$, define
\[H_{k',a}\pr{\theta} := \max_{\abs{x}\le\frac{1}{2}}\abs{\frac{E_{k'}\pr{e^{i\theta}}E_{14-k'}\pr{x+ai}}{\Delta\pr{x+ai}\pr{j\pr{x+ai}-j\pr{e^{i\theta}}}}},\]
where $a>0$, and $k'\in\prs{0,4,6,8,10,14}$.
\begin{enumerate}[label= Case \Roman*.]
    \item $y=0.75$ and $\theta \in \left[\frac{\pi}{2},1.9\right]$:\\
    By Proposition \ref{tau bounds} we get $\abs{\Delta(x+0.75i)}>0.007$. Now, 
    \begin{enumerate}[label = \roman*.]
        \item[For] $k'=0$: 
        \begin{equation*}
            H_{0,0.75}\pr{\theta} 
            \le \frac{5.25 \cdot(3.45)^2}{0.007\cdot176}<  51.31.
        \end{equation*}
        \item[For] $k'=4$:
        \begin{equation*}
            H_{4,0.75}\pr{\theta}\le \frac{1.46\cdot5.25 \cdot 3.45}{0.007\cdot176} < 21.72.
        \end{equation*}
        \item[For] $k'=6$:
        \begin{equation*}
            H_{6,0.75}\pr{\theta} \le \frac{1.99 \cdot (3.45)^2}{0.007\cdot176}<19.5.
        \end{equation*}

        \item[For] $k'=8$:
        \begin{equation*}
            H_{8,0.75}\pr{\theta} \le \frac{(1.46)^{2} \cdot 5.25}{0.007\cdot176} < 9.2.
        \end{equation*}
        
        \item[For] $k'=10$:
        \begin{equation*}
            H_{10,0.75}\pr{\theta} \le \frac{1.99\cdot 1.46 \cdot3.45}{0.007\cdot176} < 8.3.
        \end{equation*}
        
        \item[For] $k'=14$:
        \begin{equation*}
            H_{14,0.75}\pr{\theta} \le \frac{1.99\cdot(1.46)^{2}}{0.007\cdot176} < 3.5.
        \end{equation*}
    \end{enumerate}
    Therefore, for any $k'\in\{0,4,6,8,10,14\}$,
    \begin{equation}
    \max_{\abs{x}\le\frac{1}{2}}\left|\frac{E_{k'}(e^{i\theta})E_{14-k'}(x+0.75i)}{\Delta(x+0.75i)\left(j(x+0.75i)-j(e^{i\theta})\right)}\right| < e^{3.94}.
    \end{equation}
    Hence, we got $B_{1}=3.94$.\\
    
    \item $y=0.65$ and $\theta \in \left[1.9,\frac{2\pi}{3}\right]$:
    By Proposition \ref{tau bounds} we get $\abs{\Delta(x+0.65i)}>0.01$. Now,
    \begin{enumerate}[label = \roman*.]
        \item[For] $k'=0$:
        \begin{equation*}
             H_{0,0.65}\pr{\theta}\le \frac{14.26\cdot36}{311\cdot0.01}<166.7.
        \end{equation*}
        
        \item[For] $k'=4$:
        \begin{equation*}
            H_{4,0.65}\pr{\theta} \le \frac{0.9022\cdot14.26\cdot 6}{311\cdot0.01} <25.1.
        \end{equation*}

        \item[For] $k'=6$:
        \begin{equation*}
            H_{6,0.65}\pr{\theta}\le \frac{2.89\cdot 36}{311\cdot0.01} < 33.78.
        \end{equation*}

        \item[For] $k'=8$:
        \begin{equation*}
            H_{8,0.65}\pr{\theta} \le \frac{0.9022^{2} \cdot 14.26}{311\cdot0.01}<3.8.
        \end{equation*}
        
        \item[For] $k'=10$:
        \begin{equation*}
            H_{10,0.65}\pr{\theta} \le \frac{2.89\cdot 0.9022 \cdot6}{311\cdot0.01} < 5.08. 
        \end{equation*}
        
        \item[For] $k'=14$:
        \begin{equation*}
           H_{14,0.65}\pr{\theta}\le \frac{2.89\cdot0.9022^{2}}{311\cdot0.01} < 1.
        \end{equation*}
    \end{enumerate}
    Therefore, for any $k'\in\{0,4,6,8,10,14\}$,
    \begin{equation}
    \max_{\abs{x}\le\frac{1}{2}}\left|\frac{E_{k'}(e^{i\theta})E_{14-k'}(x+0.65i)}{\Delta(x+0.65i)\left(j(x+0.65i)-j(e^{i\theta})\right)}\right| < e^{5.12}.
    \end{equation}
    Hence, we got $B_{2}=5.12$.
\end{enumerate}
Finally, substituting $B_{1}=3.94$ and $B_{2}=5.12$ in \eqref{c_{2}} we obtain
\begin{multline*}
    c_{2}=\max\pr{\frac{3.94-\log\pr{1.995}}{\log\pr{10/7}},\frac{5.12-\log\pr{0.995}}{\log\pr{2}}}\\=\frac{4.04-\log\pr{1.995}}{\log\pr{10/7}}=9.11013\ldots \le 9.5,
\end{multline*}
which proves Theorem \ref{QuanMR}.
\section{Proof of Theorem \ref{the case m=1}}\label{m=1}
In this section, we investigate the zeros of the first element of the Miller basis, i.e.\ $g_{k,1}$. Using Theorem \ref{QuanMR}, we conclude that for any $k=12\ell+k'$ with $\ell> 14$, the form $g_{k,1}$ has all of its zeros in the fundamental domain on the arc $\cA$. Therefore, we need to show that the zeros of the forms $g_{k,1}$ such that $\ell \le 14$ are all on the arc $\cA$. The proof goes through Faber polynomials:

\subsection{Faber Polynomials}\label{FaberPoly}
For any nonzero modular form $f\in M_{k}$, we associate a polynomial $F_{f}\in \C[x]$ of degree $\ell-\ord_{\infty}\pr{f}$ such that $f=\Delta^{\ell}E_{k'}F_{f}(j)$, where $k=12\ell+k'$ and $k'\in\{0,4,6,8,10,14\}$, as before.
The polynomial $F_{f}$ is uniquely determined by $f$ and is called the \emph{Faber polynomial of $f$}. The valence formula \eqref{valence} implies that $f$ attains the zeros of $E_{k'}$; we call these zeros trivial. The roots of $F_{f}$ account for all the nontrivial zeros of $f$, i.e.\ for any $\tau\not\sim i,\rho$, we have $f(\tau)=0$ if and only if $F_{f}(j(\tau))=0$. Together with the fact that $j$ is injective and maps the arc $\cA=\left\{e^{i\theta}:\frac{\pi}{2}\le \theta\le \frac{2\pi}{3}\right\}$ onto the interval $[0,1728]$, we get that a form $f$ has all its zeros in the fundamental domain on the arc $\cA$ if and only if $F_{f}$ has all of its roots in the interval $[0,1728]$.

\begin{remark}
    Given some modular form $f$, it is not hard to compute its Faber polynomial (for more information see \S3 in \cite{rudnick2023}).
\end{remark}  

Hence, it is sufficient to consider the Faber polynomials of the forms $g_{k,1}$ and compute their roots.
\subsection{Some Examples}

\subsubsection{$k=48$}
The Faber polynomial of $g_{48,1}$ is:
\[F_{48,1}(t) = t^{3} -2136t^{2} + 931860t -24903328,\]
and its roots are $28.5703,\, 565.1814,\, 1542.2483 \in [0,1728]$. See Figure \ref{fig:7}.\\

\begin{figure}[ht]
    \begin{center}
    \includegraphics[height=65mm]{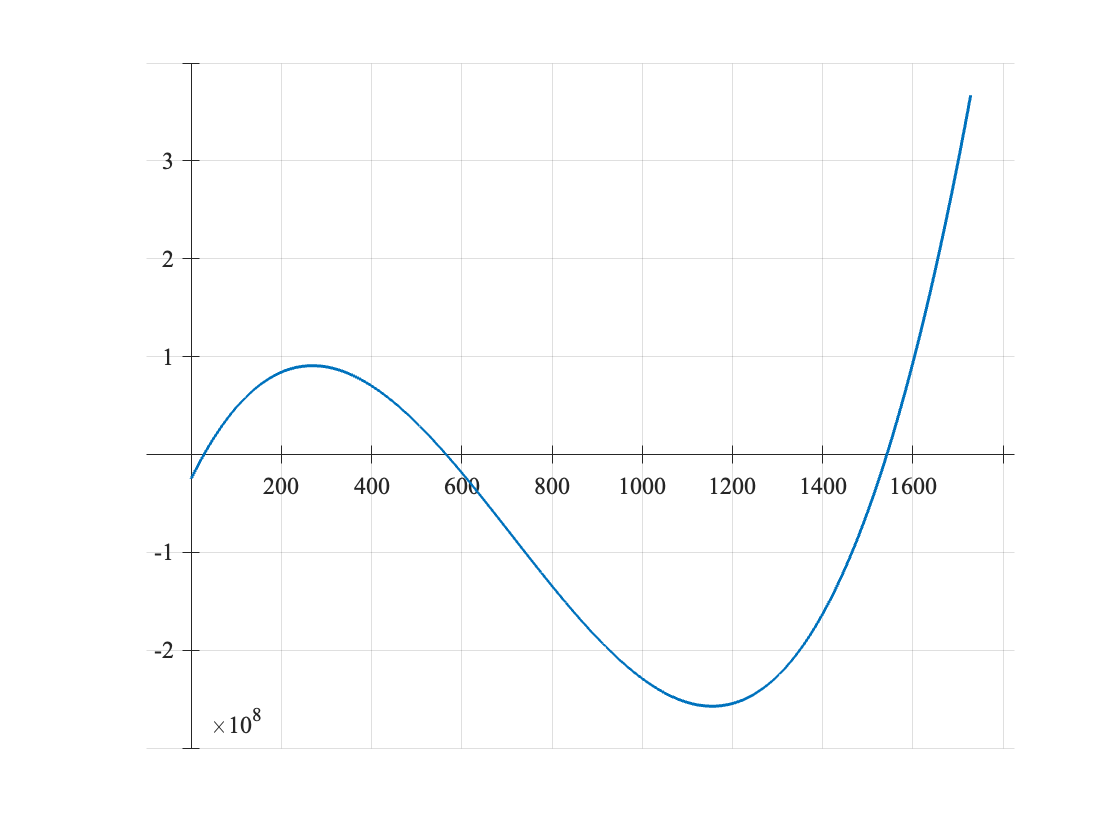}
    \caption{$F_{48,1}$ on the interval $[0,1728]$.}
     \label{fig:7}
    \end{center}
\end{figure}
\subsubsection{$k=124$}
The Faber polynomial of $g_{124,1}$ is:
\begin{multline*}
    F_{124,1}(t) = t^{9} -6696t^{8} + 18182340t^{7} -25703594848t^{6}\\+ 20207360640402t^{5}-8750844530401680t^{4}\\+1942806055074346280t^{3} -188671766710386398400t^{2} \\+ 5718177043459037019855t^{2} -21437679033112542689512.
\end{multline*}
Its roots are $4.3445$, $44.3322$, $153.6441$, $ 350.0448$, $628.6821$, $959.1844$, $1289.5802$,
$1557.8272$,  $1708.3603\in [0,1728]$. See Figures \ref{fig:8}, \ref{fig:9}.\\

\begin{figure}[ht]
    \begin{center}
    \includegraphics[height=65mm]{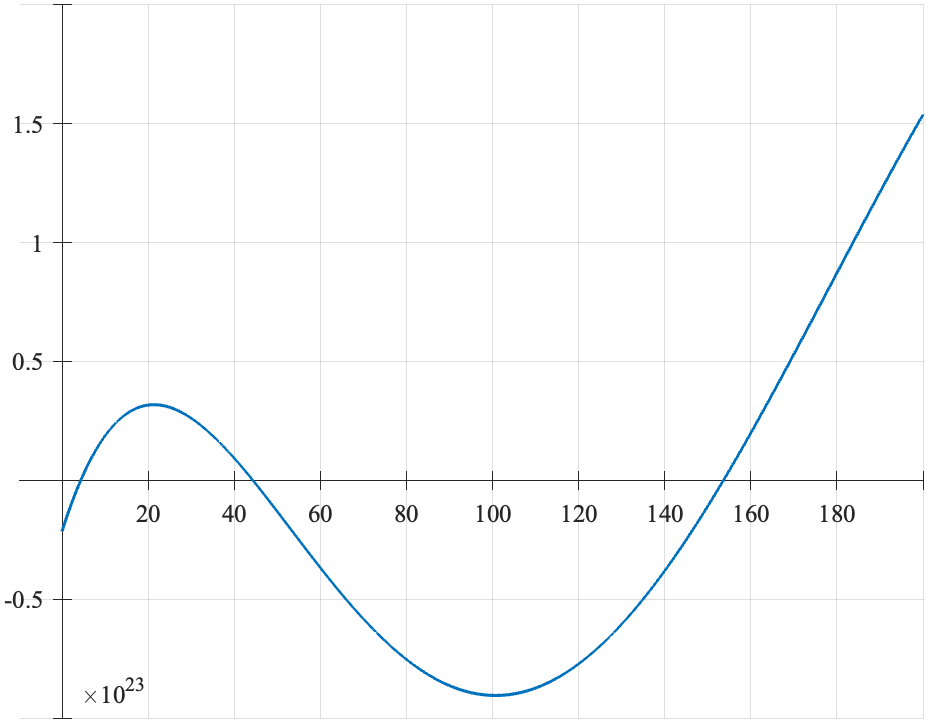}
    \caption{$F_{124,1}$ on the interval $[0,200]$.}
     \label{fig:8}
    \end{center}
\end{figure} 
\begin{figure}[ht]
    \begin{center}
    \includegraphics[height=70mm]{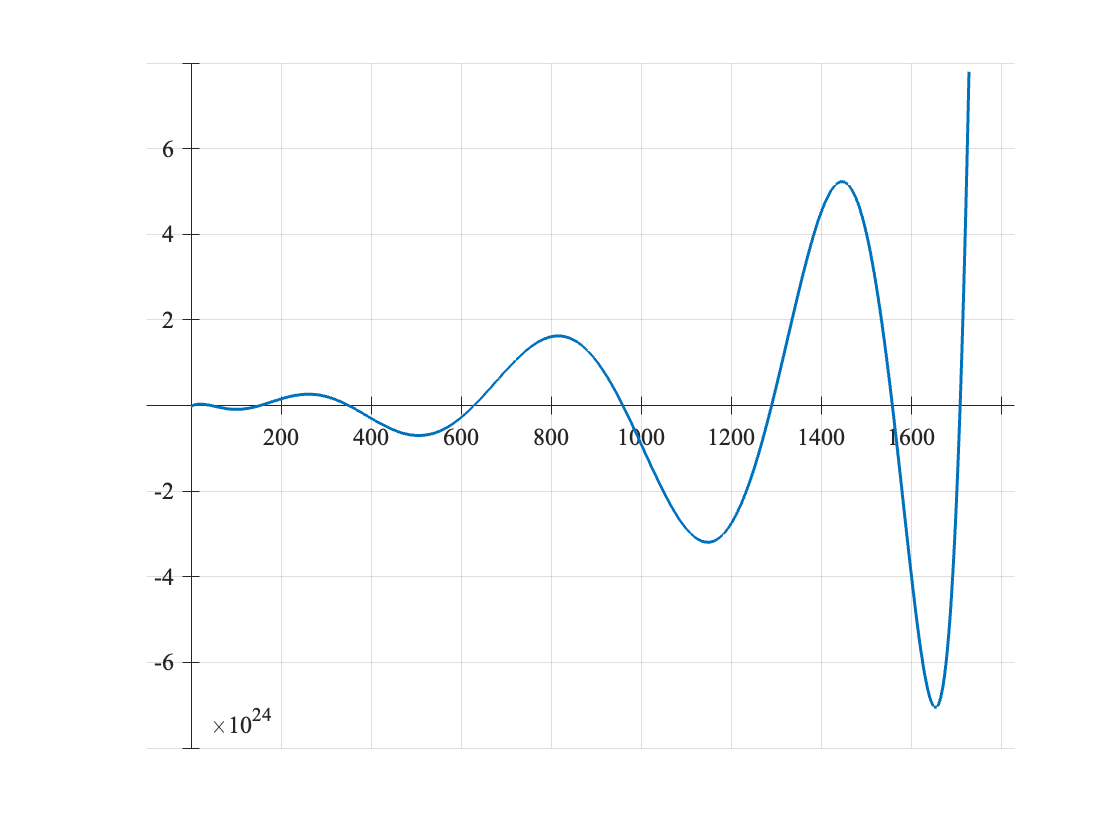}
    \caption{$F_{124,1}$ on the interval $[0,1728]$.}
     \label{fig:9}
    \end{center}
\end{figure}

Using this method, numerically computing the roots of the polynomial $F_{k,1}$ with $1\le\ell \le 14$, we can verify that for any even integer $k>14$, with $1\le\ell\le 14$ the zeros in the fundamental domain of $g_{k,1}$ are all on the arc $\cA$, which concludes the proof of Theorem \ref{the case m=1}.
\section*{Statements and Declarations}
\subsection*{Acknowledgments}
The author thanks Alon Nishry, Ze{\'e}v Rudnick, and the anonymous referees for their helpful comments. 
This research was conducted as part of the author's M.Sc.\ thesis at Tel-Aviv University under the advisement of Prof.\ Ze{\'e}v Rudnick, and supported by the European Research Council (ERC) under the European Union's Horizon 2020 research and innovation program (grant agreement No. 786758).
\subsection*{Data Availability}
We do not analyze or generate any datasets, because our work proceeds within a theoretical and mathematical approach.
\subsection*{Competing Interests}
The author declares no competing interests.

\bibliographystyle{plain}
\bibliography{Refs.bib}

\end{document}

%% file: preamble.tex
\usepackage{url}
\usepackage{xcolor}
\usepackage{xfrac}
\usepackage{wrapfig}
\usepackage{amsthm}
\usepackage{amsmath}
\usepackage{amssymb}
\usepackage{bbm}
\usepackage{physics}
\usepackage{mathtools}
\usepackage{dsfont}
\usepackage{enumitem}
\usepackage{centernot}
\usepackage{mathtools}
\usepackage{stmaryrd}
\usepackage{faktor}
\usepackage{appendix}
\usepackage{graphicx}
\usepackage{dirtytalk}
\usepackage{circuitikz}
\usepackage[nice]{nicefrac}

\makeatletter
\newcommand{\xMapsto}[2][]{\ext@arrow 0599{\Mapstofill@}{#1}{#2}}
\def\Mapstofill@{\arrowfill@{\Mapstochar\Relbar}\Relbar\Rightarrow}
\makeatother

\newcommand{\Z}{\mathbb{Z}}

\newcommand{\R}{\mathbb{R}}
\newcommand{\C}{\mathbb{C}}

\newcommand{\HH}{\mathbb{H}}

\newcommand{\cA}{\mathcal{A}}

\newcommand{\cC}{\mathcal{C}}

\newcommand{\cF}{\mathcal{F}}

\newcommand{\psl}{\text{PSL}_{2}{\left(\mathbb{Z}\right)}}

\def\res{\mathop{\text{Res}}}
\def\ord{\text{ord}}

\newcommand{\pr}[1]{{\left(#1\right)}}
\newcommand{\prb}[1]{{\left[#1\right]}}
\newcommand{\prs}[1]{{\left\{#1\right\}}}

\newcommand{\floor}[1]{{\left\lfloor #1\right\rfloor}}

\newcommand{\msum}[2]{\sum_{\begin{subarray}{c} 
   #1 \\ 
   #2  \\
 \end{subarray} }}

\newcommand{\sub}{\subset}

\numberwithin{equation}{section}

\newtheorem{thm}{Theorem}[section]
\newtheorem{conj}[thm]{Conjecture}
\newtheorem{corr}[thm]{Corollary}
\newtheorem{prop}[thm]{Proposition}
\newtheorem{lemma}[thm]{Lemma}

\newtheorem*{claim*}{Claim}

\theoremstyle{definition}

\theoremstyle{remark}
\newtheorem*{remark}{Remark}
